\documentclass[a4paper,12pt]{article}
\usepackage{amssymb}
\usepackage{amsfonts}
\usepackage{amsthm}
\usepackage{amsmath}
\usepackage[all]{xy}
\usepackage[dvips]{graphicx}
\author{Nati Linial and Doron Puder\\
Hebrew University, Jerusalem, Israel.}
\title{Word Maps and Spectra of Random Graph Lifts}

\theoremstyle{plain}  \newtheorem{ther}{Theorem}
\theoremstyle{plain}  \newtheorem{lem}[ther]{Lemma}
\theoremstyle{remark} \newtheorem{example}{Example}
\theoremstyle{plain}  \newtheorem{cor}[ther]{Corollary}
\theoremstyle{remark} \newtheorem{remark}[ther]{Remark}
\theoremstyle{plain}  \newtheorem{conj}[ther]{Conjecture}
\theoremstyle{plain}  \newtheorem{defi}[ther]{Definition}

\begin{document}
\maketitle
\date{}


\section*{Abstract}

We study here the spectra of random lifts of graphs. Let $G$ be a
finite connected graph, and let the infinite tree $T$ be its
universal cover space. If $\lambda_1$ and $\rho$ are the spectral
radii of $G$ and $T$ respectively, then, as shown by
Friedman~\cite{fri03}, in almost every $n$-lift $H$ of $G$, all
``new'' eigenvalues of $H$ are $\le
O\left(\lambda_1^{\;1/2}\rho^{1/2}\right)$. Here we improve this
bound to $O\left(\lambda_1^{\;1/3}\rho^{2/3}\right)$. It is
conjectured in~\cite{fri03} that the statement holds with the bound
$\rho + o(1)$ which, if true, is tight by~\cite{gl}. For $G$ a
bouquet with $d/2$ loops, our arguments yield a simple proof that
almost every $d$-regular graph has second eigenvalue $O(d^{2/3})$.
For the bouquet, Friedman~\cite{fri04} has famously proved the
(nearly?) optimal bound of $2\sqrt{d-1} + o(1)$.

Central to our work is a new analysis of formal words. Let $w$ be
a formal word in letters $g_{1}^{\pm 1}, \ldots ,g_{k}^{\pm 1}$.
The {\em word map} associated with $w$ maps the permutations
$\sigma_1, \ldots , \sigma_k \in S_n$ to the permutation obtained
by replacing for each $i$, every occurrence of $g_{i}$ in $w$ by
$\sigma_i$. We investigate the random variable $X_w^{(n)}$ that
counts the fixed points in this permutation when the $\sigma_i$
are selected uniformly at random. The analysis of the expectation
$\mathbb{E}(X_w^{(n)})$ suggests a categorization of formal words
which considerably extends the dichotomy of primitive vs.
imprimitive words. A major ingredient of a our work is a second
categorization of formal words with the same property. We
establish some results and make a few conjectures about the
relation between the two categorizations. These conjectures
suggest a possible approach to (a slightly weaker version of)
Friedman's conjecture.

As an aside, we obtain a new conceptual and relatively simple proof
of a theorem of A. Nica~\cite{nica}, which determines, for every
fixed $w$, the limit distribution (as $n \to \infty$) of
$X_w^{(n)}$. A surprising aspect of this theorem is that the answer
depends {\em only} on the largest integer $d$ so that $w=u^d$ for
some word $u$.

\section{Introduction}
Let $G=(V,E)$ be some fixed finite connected graph with $E= \{g_1,
\ldots, g_k\}$, and let $\lambda_1 \geq \lambda_2 \ldots \geq
\lambda_{|V|}$ be the eigenvalues of its adjacency matrix. We think
of the edges as being oriented, though the results do not depend on
the orientation chosen. We recall that $L_n(G)$ denotes the
probability space of $n$-lifts of $G$ (i.e., graphs that have an
$n$-fold cover map onto $G$). A graph $H \in L_n(G)$, has vertex set
$V \times \{1,\ldots,n\}$. For every (oriented) edge $g_i=(u,v)$, we
choose independently and uniformly a random permutation, $\sigma_i
\in S_n$, and introduce an edge between $(u,j)$ to $(v,\sigma_i(j))$
for all $j$. For background on lifts and random lifts,
see~\cite{lr,al,alm,hlw}. In particular~\cite{bl} shows how to
construct regular graph lifts with a nearly optimal spectral gap.

The projection $\pi:H\rightarrow G$ given by $\pi(u,j) = u$ is the
cover map associated with $H \in L_n(G)$. Every eigenfunction $f$
of (the adjacency matrix of) $G$ is pulled back by $\pi$ to an
eigenfunction $f\circ \pi$ of $H$. Therefore, every eigenvalue of
$G$ is also an eigenvalue of $H$. Such an eigenvalue of $H$ is
considered ``old'' while all other ones are ``new'' (including,
possibly, duplicates of the old eigenvalues corresponding to
``new'' eigenfunctions). Let $T$ be the (infinite) universal cover
of $G$ (and $H$). We consider $l_2(V(T))$, the Hilbert space of
square-summable real functions on $T$'s vertices, i.e. the
functions $f:V(T)\rightarrow \mathbb{R}$ with $\sum_{v\in
V(T)}f^2(v) < \infty$. If $A_T$ is the (infinite) adjacency matrix
of $T$, we consider the linear operator corresponding to $A_T$.
Namely, $f \rightarrow g$ where $g(v)=\sum_{u\sim v}f(u)$. This is
a bounded linear operator on $l_2(V(T))$, and we denote its
spectral radius by $\rho=\rho(T)=\rho(G)$. (Recall that, by
definition, $\rho=sup\{|\nu|~:~\nu \textrm{~is in the spectrum
of~} A_T\}$). Our main result is a new bound on $H$'s \emph{new}
eigenvalues in terms of $\lambda_1$ (the largest eigenvalue of
$G$), and $\rho$.

Let $\mu_{max}:=max\{|\mu|~:~\mu \textrm{~is a new eigenvalue of~}
H\}$. In \cite{fri03}, Friedman showed that $\mu_{max} \leq
\lambda_1^{\;1/2} \rho^{1/2} + o_n(1)$ for almost every $H$. We
improve this bound as follows:

\begin{ther} \label{ther:2/3}
Almost every random $n$-lift $H$ of $G$ satisfies:
\begin{displaymath}
\mu_{max} \leq O\left(\lambda_1^{\;1/3} \rho^{2/3}\right)
\end{displaymath}
More Specifically,
\begin{equation} \label{eq:2/3}
\mu_{max} < max\left(1,3\left(\frac{\rho}{\lambda_1}\right)^{2/3}\right)
\cdot \lambda_1^{\;1/3} \rho^{2/3} +\epsilon
\end{equation}
almost surely for every $\epsilon>0$.\\
\end{ther}

\begin{remark} \label{rem:rho<=lambda1}
Let $\Gamma$ be a (not necessarily finite) connected graph, and
let $v$ be a vertex in $\Gamma$. We denote by $t_s(v)$ the number
of closed paths of length $s$ that start and end at $v$. It is
well known that the spectral radius of $\Gamma$ equals $\limsup_{s
\to \infty}t_s(v)^{1/s}$. In particular, this value is independent
of the choice of $v$. (These facts may be proven by an easy
variation on the proof of Proposition 3.1 in \cite{buc}.)
Returning to our notation, observe that a path that starts and
ends at a vertex $v \in V(T)$ is projected to a path of the same
length that starts and ends at the corresponding vertex of $G$.
Consequently, $\rho \leq \lambda_1$ always holds.
\end{remark}

Our proof of Theorem \ref{ther:2/3} suggests an approach that may
lead to an even better (nearly optimal) bound $\mu_{max} \leq
O\left(\rho\right)$. This plan depends on an unresolved conjecture
that we present shortly. It follows from Lubotzky and
Greenberg~\cite{gl} that this statement cannot hold with any bound
smaller than $\rho -o(1)$. It is shown in~\cite{gl} that for every
infinite tree $T$ and for every $\epsilon>0$, there exists a
constant $c=c(\epsilon,T)>0$, such that if $T$ is the universal
covering space of a finite graph $\Gamma$, then at least $c
|V(\Gamma)|$ of $\Gamma$'s eigenvalues exceed $\rho(T)-\epsilon$.
Thus for $G$ fixed, and for $\epsilon>0$ there exists an
$n_\epsilon$ such that $\mu_{max}(H)
> \rho(G)-\epsilon$ for every $n \geq n_\epsilon$ and every $H \in L_n(G)$.
(Since the infinite $d$-regular tree $T_d$ has spectral radius
$\rho(T_d)=2\sqrt{d-1}$~(\cite{car}), this extends the Alon-Boppana
bound~\cite{nil} that $\lambda_2 \ge 2\sqrt{d-1} - o_n(1)$ for every
$n$-vertex $d$-regular graph).

The ``permutation model'' of random $d$-regular graphs (for $d$
even) is a special case of random lifts of graphs. In the
permutation model, $n$-vertex $d$-regular graphs are generated
through a random $n$-lift of a bouquet of $d/2$ loops. Thus, our
result, as well as Friedman's, extend earlier work on random
$d$-regular graphs. Namely, Friedman's result states that
$\lambda(G) \leq \sqrt{2d\sqrt{d-1}} + o(1)$ for almost every
$d$-regular graph, which is a slight improvement of an old result
of Broder and
Shamir~\cite{bs}. In this special case, Theorem \ref{ther:2/3}
states that $\lambda(G) = O(d^{2/3})$ holds almost surely, and the
tentative proof strategy mentioned above would yield $\lambda(G) =
O(d^{1/2})$ almost surely. In particular, we obtain the following
corollary (which is, of course, substantially weaker 
than the one proven in \cite{fri04}):
\begin{cor} \label{cor:d>=107}
If $G$ is $d$-regular and $d \ge 107$, then
$$
\mu_{max} < \lambda_1^{\;1/3} \rho^{2/3} + \epsilon 
= \left[4d(d-1)\right]^{1/3} + \epsilon.
$$
almost surely for every $\epsilon>0$.
\end{cor}

A major tool in this area is the \emph{Trace Method} which goes
back to Wigner \cite{wig}. It is based on a natural connection
between graph spectra and word-maps. This approach underlies the
work of Broder-Shamir \cite{bs} and of Friedman \cite{fri03}.

Let $w$ be a (not necessarily reduced) formal word in the letters
$g_{1}^{\pm1}, \ldots ,g_{k}^{\pm1}$. For every $k$-tuple
$(\sigma_{1}, \ldots , \sigma_{k})$ of permutations in $S_{n}$, we
form the permutation $w(\sigma_{1}, \ldots , \sigma_{k}) \in
S_{n}$, by replacing $g_{1}, \ldots ,g_{k}$ with $\sigma_{1},
\ldots ,\sigma_{k}$ in the expression of $w$. For instance, if
$w=g_2 g_1^{\;2} g_2^{\;-1} g_3$, then $w(\sigma_1, \sigma_2,
\sigma_3) = \sigma_2 \sigma_1^{\;2} \sigma_2^{\;-1} \sigma_3$. The
correspondence between $w$ and the permutation $w(\sigma_{1},
\ldots , \sigma_{k}) \in S_{n}$ is called a {\em word map}. Such
maps can be evaluated in groups other than $S_n$ as well (we refer
to this briefly in Section \ref{sec:word-maps}). The study of word
maps has a long history in group theory (see \cite{lsh} and the
references therein). Our perspective is mostly combinatorial and
probabilistic.\\

For fixed formal word $w$ we denote by $X_w^{(n)}$ a random variable
on $S_{n}^{\;k}$ which is defined by:
\begin{equation} \label{eq:xwn}
X_w^{(n)}(\sigma_{1},\ldots,\sigma_{k}) = \textrm{\# of fixed points
of } w(\sigma_{1}, \ldots , \sigma_{k}).
\end{equation} \\

Now let $H$ be an $n$-lift of $G$ and let $A_G, A_H$ be the
adjacency matrices of $G,H$ resp. We denote by $\mu$ a running
index for the ``new'' eigenvalues of $H$. For every $t\geq1$, the
trace of $A_H^{\;~t}$ equals the number of closed paths of length
$t$ in $H$. This number can also be expressed as $\left(\sum_\mu
\mu^t \right) + \left(\sum_{i=1}^{|V(G)|}\lambda_i^{\;t}\right)$.
Therefore, for $t$ even we obtain:
\begin{displaymath}
\mu_{max}^{\;~~~t} \leq \sum_\mu \mu^t = \left(\sum_\mu \mu^t +
\sum_{i=1}^{|V(G)|}\lambda_i^{\;t}\right) -
\sum_{i=1}^{|V(G)|}\lambda_i^{\;t} = tr(A_H^{\;~t}) - tr(A_G^{\;~t})
\end{displaymath}

Every closed path in $H$ is a lift of a closed path in $G$. Since
the edges of $G$ are labeled $g_1,\ldots,g_k$, every (closed) path
in $G$ corresponds to some formal word $w$ in $g_1^{\pm
1},\ldots,g_k^{\pm 1}$. The closed lifts of this path are in $1:1$
correspondence with the fixed points of $w(\sigma_{1}, \ldots ,
\sigma_{k})$, so that their number is $X_w^{(n)}(\sigma_{1},
\ldots , \sigma_{k})$. Let ${\cal CP}_t(G)$ denote the set of all
closed paths of length $t$ in $G$ (i.e., $|{\cal CP}_t(G)|=
tr(A_G^{\;t})$). The above inequality now becomes:

\begin{displaymath}
\mu_{max}^{\;~~~t} \leq tr(A_H^{\;~t}) - tr(A_G^{\;~t}) = \sum_{w
\in {\cal CP}_t(G)} \left[X_w^{(n)}(\sigma_{1}, \ldots , \sigma_{k})
- 1 \right]
\end{displaymath}

Taking expectations, we obtain:
\begin{equation} \label{eq:trace}
\mathbb{E}(\mu_{max}^{\;~~~t}) \leq
\sum_{w \in {\cal CP}_t(G)} \left[\mathbb{E}(X_w^{(n)}) - 1
\right]
\end{equation}

Equation \eqref{eq:trace} shows the significance of
$\mathbb{E}(X_w^{(n)}) - 1$ in the study of spectra in random lifts.
If we let $\Phi_w(n)$ equal $\frac{\mathbb{E}(X_w^{(n)}) - 1}{n}$,
it turns out that for every $w$, $\Phi_w$ can be expressed as a
power series in $\frac{1}{n}$. Namely,

\begin{equation} \label{eq:Phi}
\Phi_w(n) = \frac{\mathbb{E}(X_w^{(n)}) - 1}{n} = \sum_{i=0}^\infty
a_i(w)\frac{1}{n^i},
\end{equation}
where the $a_i(w)$ are integers. (This fact appears in \cite{nica},
but we present (Lemma \ref{lem:Phi}) a new and simpler proof). This
induces a categorization of words in $\mathbf{F}_k$, the free group
with generators $g_1,\ldots,g_k$. Namely, $\phi(w)$ is the smallest
index $i$ for which $a_i(w) \neq 0$, or $\infty$ if
$\mathbb{E}(X_w^{(n)}) \equiv 1$.\\

We consider next (Section~\ref{sbs:beta}) another categorization
of the words in $\mathbf{F}_k$, which does not depend on a
word-map to specific groups such as $S_n$. To every $w \in
\mathbf{F}_k$ we associate $\beta(w)$ which is a non-negative
integer or $\infty$. The categorizations induced by both $\phi(w)$
and $\beta(w)$ extend the dichotomy between primitive and
imprimitive words (Recall that $w$ is called \emph{imprimitive} if
$w=u^{d}$ for some $u \in \mathbf{F}_{k}$ and $d\geq2$). Without
going into the (somewhat lengthy) definition, let us say that the
main step in both \cite{bs} and \cite{fri03} can be viewed as the
observation that for $i=0,1$, $\phi(w)=i$ iff $\beta(w)=i$. Our
aforementioned conjecture states in this language that
$\phi(w)=\beta(w)$ for every word $w$ (Conjecture
\ref{conj:beta=phi}). These relations between $\phi$ and $\beta$
allow us to bound the sum in the r.h.s of \eqref{eq:trace}: We can
bound the contribution of $w$ to this sum in terms of $\phi(w)$.
This is complemented by bounding the number of words $w \in {\cal
CP}_t(G)$ with a given value of $\beta(w)$ which bound is stated
in terms of $\rho$. Indeed, a key step in the present paper (Lemma
\ref{lem:beta>=3}) can be interpreted as a partial proof of the
claim that $\beta(w)=2$ iff $\phi(w)=2$.

As an aside to our work we obtain a new conceptual and relatively
simple proof of a theorem of A. Nica~\cite{nica}, which determines
for every fixed $w$ the limit distribution of $X_w^{(n)}$ as $n
\to \infty$ (see Theorem \ref{ther:limit-xwln}). We carry out a
similar analysis for all higher moments of $X_w^{(n)}$, and use
the method of moments to derive Nica's result. A surprising aspect
of this theorem is that the limit distribution depends {\em only}
on the largest integer $d$ such that $w=u^d$ for some $u \in
\mathbf{F}_{k}$. Nica's full result (which we derive by the same
argument) concerns not only fixed points but applies just as well
to the number of $L$-cycles for any fixed $L \geq 1$.

The paper is arranged as follows. We begin (Section
\ref{sec:word-maps}) with our analysis of word maps and introduce
the two new categorizations of formal words. Based on this
analysis, we prove Theorem \ref{ther:2/3} in Section
\ref{sec:lifts}. In Section \ref{sec:nica} we deal with the
distribution of the number of $L$-cycles in $w(\sigma_1,\ldots,
\sigma_k)$ and present our new proof for Nica's Theorem. For the
reader interested only in this new proof, this section is mostly
self-contained with only occasional references to earlier parts of
the paper. There are numerous open problems and conjectures raised
in this paper, some of which we collect in
Section~\ref{sec:open_prob}.

\section{Word Maps and the Level of Primitivity of a Word}
\label{sec:word-maps}

We begin with some notation. We denote by $\Sigma_k$ the set of
all finite words in letters $g_{1}^{\pm 1}, \ldots ,g_{k}^{\pm 1}$
(though we occasionally use the letters $a,b,c,\ldots$ instead).
The quotient of $\Sigma_k$ modulo reduction of words is
$\mathbf{F}_k$, the set of elements of the free group on $k$
generators. For instance, the set ${\cal CP}_t(G)$ introduced
before Equation \eqref{eq:trace} is a subset of $\Sigma_k$, so it
may contain different words which are equivalent as members of
$\mathbf{F}_k$.

For every group $P$ and every word $w\in\Sigma_k$, the \emph{word
map} $w:P^k \rightarrow P$ is defined by substitutions and
composition. For $p_1,\ldots,p_k \in P$, the element
$w(p_1,\ldots,p_k)$ is obtained by substituting $p_i$ for each
occurrence of $g_i$ in $w$ and this for every $1 \le i \le k$.
Clearly, the word map of $w$ is invariant under reductions, so we
can regard $w$ as an element in $\mathbf{F}_k$.

Most research on word maps concerns the range of certain fixed
words $w$ in a group $P$. More specifically, for $P$ finite, it is
of interest to understand the distribution induced on $P$ by the
word map $w:P^k \rightarrow P$ and the uniform distribution on
$P^k$. This perspective makes it natural to consider an
equivalence relation on words (beyond that of reduction).

In order to introduce this equivalence relation, we now recall
some simple terminology from combinatorial group theory. There are
three \emph{elementary Nielsen transformations} defined on the
free group $\mathbf{F}_k$: (i) Exchanging any two generators $g_i$
and $g_j$ for some $i \neq j$, (ii) Replacing some $g_i$ with
$g_i^{-1}$, (iii) Replacing any $g_i$ by $g_ig_j$, for some $i
\neq j$. We recall (e.g.~\cite{mks}, Theorem 3.2) that these
transformations generate the automorphism group $\mathbf{A}_k$ of
$\mathbf{F}_k$. We say that two words $w_1, w_2 \in \mathbf{F}_k$
are equivalent, and denote $w_1 \sim w_2$, if they belong to the
same orbit of $\mathbf{A}_k$. Obviously, $``\sim"$ is an
equivalence relation. It is quite clear that for every finite
group $P$, every two equivalent words $w_1, w_2 \in \mathbf{F}_k$
induce the same distribution on $P$. We do not know whether
the converse is true as well, and we state a specific problem
(Conjecture~\ref{nielsen_conj}) in this vein.

Given a word $w$ and the distribution it induces on a group $P$,
it is of interest to consider how far this distribution is from
the uniform distribution. The two gradings of words $\phi(\cdot)$
and $\beta(\cdot)$ can be viewed as our attempts to capture this
intuition. Both parameters associate a non-negative integer or
$\infty$ with every $w \in \mathbf{F}_k$, and they tend to grow
as the above-mentioned distance decreases. Of course, the distribution
furthest away from the uniform distribution corresponds to the
word $w=1$. Indeed, $\beta(w) = 0$ iff $\phi(w) = 0$ iff $w=1$
(Lemma~\ref{lem:phi=beta=0}).
Also, $\beta(w) = 1$ iff $\phi(w) = 1$ iff $w$ is imprimitive
(Lemma~\ref{lem:phi=beta=1}). In
this case the range of $w$ contains only powers of certain
exponent. (Recall that $w \in \mathbf{F}_k$ is called imprimitive
if $w=u^d$ for some word $u$ and $d \geq 2$.) At the other end of
the scale, both $\beta(w)$ and $\phi(w)$ equal $\infty$ for words
that are $\sim$-equivalent to a single-letter word. Clearly, such
words always induce the uniform distribution on $P$. Another
important property is that both $\beta$ and $\phi$ are invariant
under $``\sim"$.

The definition of $\phi(w)$ depends on the word map for the
symmetric group $S_n$ (see Section \ref{sbs:exp-fix-points}). The
definition of $\beta(w)$ is more involved and is based on a
certain analysis of $w$ as a formal word without reference to
groups (see Section \ref{sbs:beta}). In fact, we have arrived at
our definition of $\beta(w)$ through our study of $\phi(w)$. Some
proven results and extensive numerical simulations suggest that
$\phi(w)=\beta(w)$ for every $w$ (Section \ref{sbs:phi=beta}). One
advantage of the parameter $\beta$ over $\phi$ is that we can
bound the number of words with fixed value of $\beta$ - see
Section \ref{sbs:n_words_with_beta=}.

Since both $\phi(w)$ and $\beta(w)$ offer an extension of the
primitive-imprimitive dichotomy for words, we tend to think of
them as quantifying \emph{``the level of primitivity''} of a word
(this level is 0 if $w=1$, it is 1 if $w$ is imprimitive, and
$\geq 2$ for primitive words - see Section \ref{sbs:phi=beta}).

\subsection{The Word Map $w:S_n^{\;k} \rightarrow S_n$ and $\phi(\cdot)$}
\label{sbs:exp-fix-points}

In this section we present a method to calculate the expectation
of $X_w^{(n)}$, the number of fixed points in $w(\sigma_{1},
\ldots , \sigma_{k})$ (defined in \eqref{eq:fixed-points-exp}). We
count the fixed points in $w(\sigma_1,\ldots,\sigma_k)$ for
\emph{all} $k$-tuples $(\sigma_1,\ldots,\sigma_k) \in S_n^{\;k}$
and divide by $(n!)^k$. This calculation is carried out through a
certain categorization of all fixed points. We note that similar
considerations appear in~\cite{nica} and in~\cite{fri91}.

We begin with some technicalities. Let
$w=g_{i_{1}}^{\;\alpha_{1}}g_{i_{2}}^{\;\alpha_{2}}\ldots
g_{i_{m}}^{\;\alpha_{m}} \in \Sigma_k$, where $i_1,\ldots,i_m \in
\{1,\ldots,k\}$ and $\alpha_1,\ldots,\alpha_m \in \{-1,1\}$, and let
$\sigma_1, \ldots, \sigma_k \in S_n$. Assume that $s_0 \in
\{1,\ldots,n\}$ is a fixed point of $w(\sigma_1,\ldots,\sigma_k)$.
Associated with $s_0$ is the following closed trail:
\begin{displaymath}
s_0 \stackrel{\sigma_{i_1}^{\;\alpha_1}}{\longrightarrow} s_1
\stackrel{\sigma_{i_2}^{\;\alpha_2}}{\longrightarrow} s_2
\stackrel{\sigma_{i_3}^{\;\alpha_3}}{\longrightarrow} \ldots
\stackrel{\sigma_{i_{m-1}}^{\;\alpha_{m-1}}}{\longrightarrow}
s_{m-1} \stackrel{\sigma_{i_m}^{\;\alpha_m}}{\longrightarrow} s_0
\end{displaymath}
with $s_1,\ldots,s_{m-1} \in \{1,\ldots,n\}$, and $s_{b ~
\textrm{mod} ~ m} = \sigma_{i_b}^{\;\alpha_b}(s_{b-1})$
($b=1,\ldots,m$).

Note that for the sake of convenience, we compose permutations from
left to right. This is inconsequential for the analysis of the
variables $X_w^{(n)}$ since $w(\sigma_1,\ldots,\sigma_k)$ with
left-to-right composition is the inverse of
$w(\sigma_1^{\;-1},\ldots,\sigma_k^{\;-1})$ with right-to-left
composition, and thus both have the same cycle structure.

We categorize fixed points according to their associated trails. Let
$s_0\to\ldots\to s_{m-1} \to s_0$ and $s'_0\to\ldots\to s'_{m-1} \to
s'_0$ be the trails of the fixed points $s_0$ and $s'_0$ in
$w(\sigma_1,\ldots \sigma_k)$ and $w(\sigma'_1,\ldots,\sigma'_k)$
respectively. These two trails are placed in \emph{the same
category}, if they have the same coincidence pattern, that is, if
for every $i,j\in\{0,\ldots,m-1\}$, $s_i=s_j \Leftrightarrow
s'_i=s'_j$.

Each closed trail consists of $m$ integers, or points, possibly with
repetitions, and each category of trails uniquely corresponds to
some \emph{partition} of these $m$ points. Consequently, there are
at most $B(m)$ categories, where the $m$-th Bell Number, $B(m)$, is
the number of partitions of an $m$ element set. This bound is,
however, not tight. For instance, if $s_{a+1}=\sigma_j(s_a)$ and
$s_{b+1}=\sigma_j^{\;-1}(s_b)$, then $s_a = s_{b+1} \Leftrightarrow
s_{a+1} = s_b$. Therefore, not every partition corresponds to a
realizable category of trails.

It is convenient to associate a directed edge-colored graph $\Gamma$
with each category. Vertices in $\Gamma$ correspond to blocks in the
partition that defines $\Gamma$'s category. In other words, $\Gamma$
has as many vertices as the number of distinct integers among the
$s_b$'s. There is a directed edge labeled $j$ from one vertex
(=block) to another, whenever the trails include an arrow labeled
$\sigma_j$ (resp. $\sigma_j^{\;-1}$) from a point in the first
(second) block to a point in the second (first) one.

Of special importance is the graph associated with the finest
possible partition which we call \emph{the universal graph}. Two
points in the trail are merged in this partition if and only if they
are merged in every realizable partition. If $w$ is cyclically
reduced (i.e., no two consecutive letters are inverses, nor are the
first and last letter), this is the partition where all $m$ points
in the trail are distinct. To illustrate, we draw in Figure
\ref{fig:univ-graphs} the universal graph of three different words.

\begin{figure}[htb]
\centering
\includegraphics[width=0.9\textwidth]{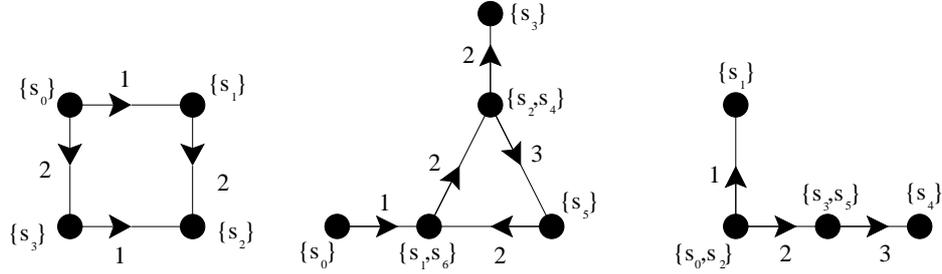}
\caption{From left to right: the universal graphs of
$w=g_1g_2g_1^{\;-1}g_2^{\;-1}$ (the commutator word), of $w=g_1 g_2
g_2 g_2^{\;-1} g_3 g_2 g_1^{\;-1}$, and of $w=g_1 g_1^{\;-1} g_2 g_3
g_3^{\;-1} g_2^{\;-1}$.} \label{fig:univ-graphs}
\end{figure}

All other graphs are now easily derived as quotients of the
universal graph, or partitions of its vertices. A quotient graph has
one vertex per each block in the partition. It has a $j$-labeled
directed edge ($j$-edge for short) from block $v_1$ to block $v_2$,
if the universal graph contains a $j$-edge from a vertex in $v_1$ to
a vertex in $v_2$. A quotient is not \emph{realizable} if it contains
two distinct $j$-edges with common head and different tails or
vice-versa. We denote by ${\cal Q}_w$ the set of all realizable
quotients .

To illustrate, we draw (Figure \ref{fig:7graphs}) all the realizable
quotient graphs of the universal graph of the commutator word (one
of the graphs in Figure \ref{fig:univ-graphs}). Note that a four
element set has 15 partitions (the fourth Bell number, $B(4)= 15$),
of which only 7 are realizable in this case.

\begin{figure}[htb]
\centering
\includegraphics[width=0.8\textwidth]{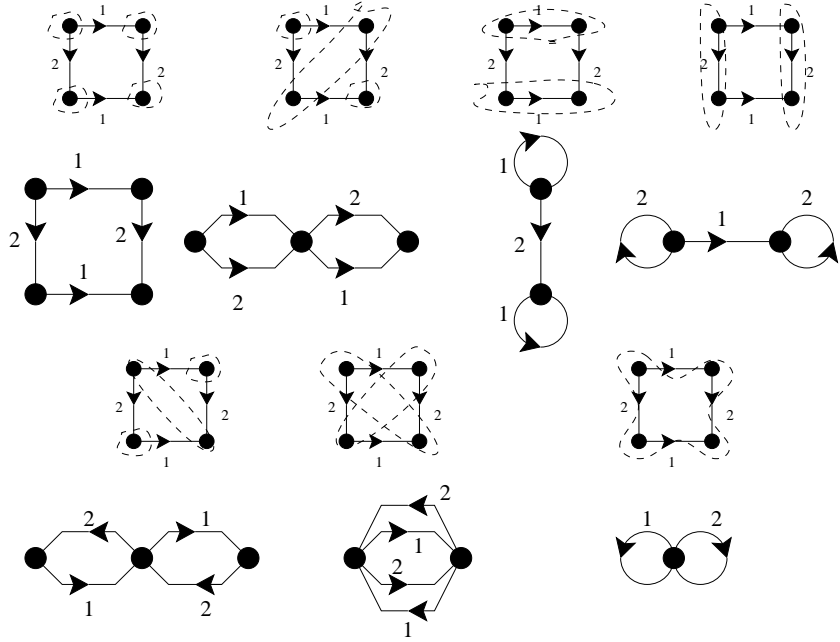}
\caption{The 7 different graphs in ${\cal Q}_w$, the set of graphs
of categories of fixed points when $w=g_1g_2g_1^{\;-1}g_2^{\;-1}$.
Each graph is drawn together with the partition of the universal
graph which yields it. (We do not specify the block corresponding to each vertex).} \label{fig:7graphs}
\end{figure}

These graphs suggest a simple formula for the number of fixed points
in each category. Let $v_\Gamma$, ($e_\Gamma$) be the number of
vertices (edges) in the graph $\Gamma$, and $e_\Gamma^j$ be the
number of $j$-edges (and so $e_\Gamma=\sum_{j=1}^k e_\Gamma^j$). To
count the number of fixed points in $\Gamma$'s category, or the
number of \emph{realizations} of $\Gamma$, we first label $\Gamma$'s
vertices by distinct numbers from $\{1, \ldots, n\}$ (i.e., specify
the values of $s_0,\ldots,s_{m-1}$) . This can be done in
$n(n-1)\ldots (n-v_\Gamma+1)$ ways. For each $j=1,\ldots,k$ there
are $\big(n-e_\Gamma^j\big)!$ permutations that are consistent with
the $e_\Gamma^j$ values in the permutation $\sigma_j$ that are
already determined. Thus, the number of realizations of $\Gamma$ is:

\begin{equation} \label{eq:n-realizations}
N_\Gamma(n) =
n(n-1)\ldots(n-v_\Gamma+1)\prod_{j=1}^{k}\big[n-e^j_\Gamma\big]!
\end{equation}

A formula for the expectation of $X_w^{(n)}$ is now at hand:

\begin{eqnarray} \label{eq:fixed-points-exp}
\mathbb{E}(X_w^{(n)}) &=& \frac{1}{(n!)^k}
\sum_{\sigma_1,\ldots,\sigma_k \in S_n}
X_w^{(n)}(\sigma_1,\ldots,\sigma_k)  = \frac{1}{(n!)^k} \sum_{\Gamma
\in {\cal Q}_w}N_\Gamma(n) = \nonumber \\
&=& \sum_{\Gamma \in {\cal Q}_w}
\frac{n(n-1)\ldots(n-v_\Gamma+1)}{\prod_{j=1}^k n(n-1)\ldots
(n-e_\Gamma^j+1)} = \nonumber \\
&=& \sum_{\Gamma \in {\cal Q}_w}
\left(\frac{1}{n}\right)^{e_\Gamma-v_\Gamma}
\frac{\prod_{l=1}^{v_\Gamma-1}(1-\frac{l}{n})}
{\prod_{j=1}^k\prod_{l=1}^{e_\Gamma^j-1}(1-\frac{l}{n})}
\end{eqnarray}
(The third equality holds only for $n$ that is $\ge e_\Gamma^j$ for
all $j$ and $\Gamma$.)

We illustrate these calculations for $w$ the commutator word
$g_1g_2g_1^{\;-1}g_2^{\;-1}$. If we go over the graphs in Figure
\ref{fig:7graphs} in clockwise order starting at the upper-left
graph, \eqref{eq:fixed-points-exp} becomes:
\begin{eqnarray*}
\mathbb{E}(X_w^{(n)}) =&& \frac{n(n-1)(n-2)(n-3)}{[n(n-1)][n(n-1)]}
+
\frac{n(n-1)(n-2)}{[n(n-1)][n(n-1)]} + \\
&& + \frac{n(n-1)}{[n(n-1)][n]} + \frac{n(n-1)}{[n][n(n-1)]} +
\frac{n}{[n][n]} + \\
&& + \frac{n(n-1)}{[n(n-1)][n(n-1)]} +
\frac{n(n-1)(n-2)}{[n(n-1)][n(n-1)]} = \frac{n}{n-1}
\end{eqnarray*}

In \eqref{eq:Phi} we defined $\Phi_w(n)$ to be
$\frac{\mathbb{E}(X_w^{(n)}) - 1}{n}$. The following lemma states an
important property of $\Phi_w$:

\begin{lem} \label{lem:Phi}
Associated with every $w \in \mathbf{F}_k$ is a power series
$\sum_{i=0}^\infty a_i(w) x^i$ that has a positive radius of
convergence where the coefficients $a_i(w)$ are integers. In
particular, for every $w$ and for every sufficiently large $n$,
there holds $\Phi_w(n) = \sum_{i=0}^\infty a_i(w)\frac{1}{n^i}$.
\end{lem}

\begin{proof}
As we saw in \eqref{eq:fixed-points-exp}, for large enough $n$
(e.g., $n \geq |w|$ suffices),
\begin{equation} \label{eq:Phi-series}
\Phi_w(n) = -\frac{1}{n} + \sum_{\Gamma \in {\cal Q}_w}
\left(\frac{1}{n}\right)^{e_\Gamma-v_\Gamma+1}
\frac{\prod_{l=1}^{v_\Gamma-1}(1-\frac{l}{n})}
{\prod_{j=1}^k\prod_{l=1}^{e_\Gamma^j-1}(1-\frac{l}{n})}.
\end{equation}
Since every $\Gamma$ is a connected graph, we have
$e_\Gamma-v_\Gamma+1 \geq 0$. The lemma follows when we individually
consider the expression corresponding to each $\Gamma$. (See the
proof of Lemma \ref{lem:analyze-o(1)} for a thorough analysis of
these expressions).
\end{proof}

By Lemma \ref{lem:Phi}, the contribution of every $w \in {\cal
CP}_t(G)$ in \eqref{eq:trace}, is $\frac{a_i(w)+o(1)}{n^{i-1}}$ for
some nonnegative integer $i$. This induces the following useful
grading of words:
\begin{equation} \label{eq:phi}
\phi(w) := \left\{ \begin{array}{ll} \textrm{the smallest integer
$i$ with }a_i(w) \neq 0 & \textrm{ if } \mathbb{E}(X_w^{(n)}) \not
\equiv 1 \\
\infty & \textrm{ if } \mathbb{E}(X_w^{(n)}) \equiv 1
\end{array} \right.
\end{equation}

Recall that although the construction of $\Phi_w$ depends on the
actual representation of $w$ (a reduction of $w$ usually changes
${\cal Q}_w$), this function captures some features of the
distribution of the image of $w$ on $S_n$. Thus $\Phi_w$, as well
as $\phi(w)$, are invariant not only under reduction, but also
under $``\sim"$.

With this new terminology we can reinterpret both \cite{bs} and
\cite{fri03} as follows: Both papers rely on the fact that
$\phi(w)=0$ iff $w$ reduces to the empty word, and that
$\phi(w)=1$ iff $w$ is imprimitive (see Lemmas
\ref{lem:phi=beta=0} and \ref{lem:phi=beta=1}). To study
the new spectrum of random lifts, one proceeds as follows:
The number of words of these two kinds can be bounded in terms of
$\rho$ (the spectral radius of the universal cover) alone (and does not depend on $\lambda_1$, the spectral radius of the
base graph). Finally, the rest of the words (which are, in
fact, the vast majority) contribute to the summation in
\eqref{eq:trace} only $O\left(\frac{1}{n}\right)$ each.

Here we extend these ideas and seek (with partial success) a
similar characterization for all words with $\phi(w)=i$ for fixed
$i$. Our analysis of the new spectrum extends these arguments and
refines them. We further split the above-mentioned third set and
attain an improved bound on the contributions of these
subsets to the sum in Equation~\eqref{eq:trace}.

\subsection{More on the Level of Primitivity: $\beta(\cdot)$}
\label{sbs:beta}

We now start our second attempt at capturing the ``level of
primitivity" of $w$ by means of the parameter $\beta(w)$. We begin
with some definitions.

Recall the notion of a trail from Section \ref{sbs:exp-fix-points}.
Consider then the following trail through $w =
g_{i_{1}}^{\;\alpha_{1}} g_{i_{2}}^{\;\alpha_{2}} \ldots
g_{i_{|w|}}^{\;\alpha_{|w|}}$ (this time, the $s_i$'s should not be
thought of as numbers but rather as abstract symbols, and the trail
deliberately ends with $s_{|w|}$ and not with $s_0$):
\begin{displaymath}
s_0 \stackrel{g_{i_1}^{\;\alpha_1}}{\longrightarrow} s_1
\stackrel{g_{i_2}^{\;\alpha_2}}{\longrightarrow} s_2
\stackrel{g_{i_3}^{\;\alpha_3}}{\longrightarrow} \ldots
\stackrel{g_{i_{|w|-1}}^{\;\alpha_{m-1}}}{\longrightarrow} s_{|w|-1}
\stackrel{g_{i_{|w|}}^{\;\alpha_m}}{\longrightarrow} s_{|w|}
\end{displaymath}

As in our former definition of a realizable category of trails, we
say that a partition of $\{s_0,\ldots,s_{|w|}\}$ is realizable if
the following conditions hold: Whenever $i_h=i_l$ and
$\alpha_h=\alpha_l$, $s_{h-1}$ is in the same block with $s_{l-1}$
(we denote $s_{h-1} \equiv s_{l-1}$) iff $s_h \equiv s_l$. Likewise,
whenever $i_h=i_l$ and $\alpha_h=-\alpha_l$, $s_{h-1} \equiv s_l
\Leftrightarrow s_h \equiv s_{l-1}$. (In other words, a partition
is realizable whenever it traces a trail of some point, fixed or not,
through $w(\sigma_1,\ldots,\sigma_k)$ for some $\sigma_1,\ldots,\sigma_k \in S_n$
and some $n$.)

As before, to each realizable partition of $\{s_0,\ldots,s_{|w|}\}$
corresponds a directed edge-colored graph $\Gamma$, which is a
quotient of the graph of the trail. According to our former
notation, $\Gamma \in {\cal Q}_w$ whenever $s_0 \equiv s_{|w|}$. We
now concentrate on the number of pairs of $s_i$'s that should be
merged in order to yield a specific $\Gamma$.

\begin{defi} \label{df:generating-set}
Let $\Gamma$ be the quotient graph corresponding to some realizable
partition of $\{s_0,\ldots,s_{|w|}\}$. We say that set of pairs
$\left\{\{s_{j_1},s_{k_1}\},\ldots,\{s_{j_r},s_{k_r}\}\right\}$
\textbf{generates} $\Gamma$, if $\Gamma$ corresponds to the finest
realizable partition in which $s_{j_i} \equiv s_{k_i}~
\forall i=1,\ldots,r$.\\
\end{defi}

\begin{example}
Let us return to the commutator word $w = g_1 g_2 g_1^{\;-1}
g_2^{\;-1}$. The trail here is \[s_0 \stackrel{g_1}{\longrightarrow}
s_1 \stackrel{g_2}{\longrightarrow} s_2
\stackrel{g_1^{\;-1}}{\longrightarrow} s_3
\stackrel{g_2^{\;-1}}{\longrightarrow} s_4.\] In Figure
\ref{fig:minimal-gps} we revisit the seven quotient graphs from
Figure \ref{fig:7graphs} (the seven graphs in ${\cal Q}_w$), and
specify a smallest generating set for each of them.

\begin{figure}[htb]
\centering
\includegraphics[width=0.8\textwidth]{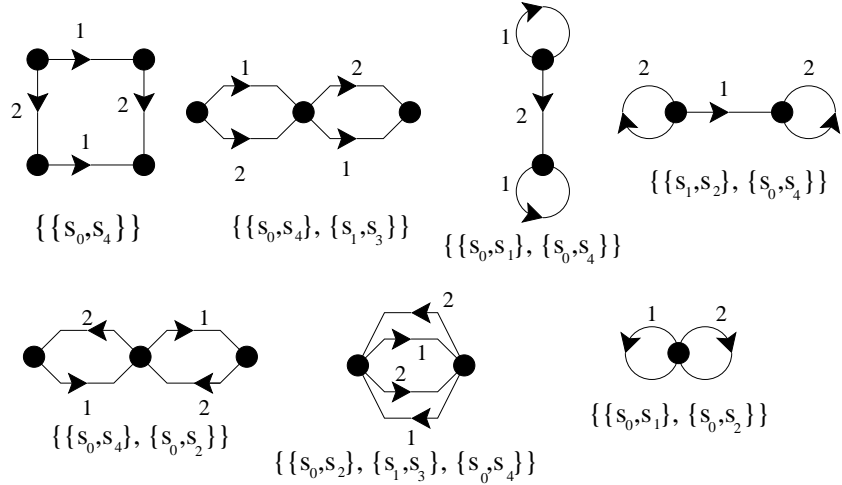}
\caption{A smallest generating set of each of the seven graphs in
${\cal Q}_w$, the set of graphs of categories of fixed points when
$w=g_1g_2g_1^{\;-1}g_2^{\;-1}$.} \label{fig:minimal-gps}
\end{figure}

\end{example}

We denote by $\chi(\Gamma) = e_\Gamma-v_\Gamma+1$
the Euler characteristic of $\Gamma$.
It turns out that there is a tight
connection between $\chi(\Gamma)$ and the smallest size of a
generating set of $\Gamma$:

\begin{lem}\label{lem:char=mgps}
Let $\Gamma$ be the quotient graph corresponding to some realizable
partition of $\{s_0,\ldots,s_{|w|}\}$. The smallest cardinality of a
generating set for $\Gamma$ is $\chi(\Gamma)$.
\end{lem}

\begin{proof}
It is quite easy to construct a set $\widehat{S}$ of
$\chi(\Gamma)$ pairs that generates $\Gamma$. To this end, we
adopt the original terminology of~\cite{bs}. As we follow the path
of $w$ through $\Gamma$, each move has one of three types. In a
\textbf{free} step we traverse a new edge and reach a new vertex,
and so one vertex and one edge are added to the partial graph. In
a \textbf{coincidence} a new edge leads us to an ``old'' vertex,
so we gain one new edge and no new vertices. In a \textbf{forced}
step we traverse an old edge (necessarily to an old vertex), so
the numbers of vertices and edges remain unchanged. Consequently,
$\chi(\Gamma)$ equals the number of coincidences in this walk. We
introduce into $\widehat{S}$ one pair for each coincidence. If the
$j$-th step is a coincidence in which we reach a vertex in
$\Gamma$ representing the block $s_{i_1},\ldots,s_{i_r}~(i_1 \leq
\ldots \leq i_r < j)$, we add $\{s_j,s_{i_1}\}$ to $\widehat{S}$.
Clearly, the cardinality of $\widehat{S}$ is $\chi(\Gamma)$ and it
generates $\Gamma$.

To see that $\Gamma$ has no generating set smaller than
$\widehat{S}$, consider ${\cal S}$, the collection of all
generating sets of $\Gamma$ of the smallest possible cardinality.
We claim that $\widehat{S}$ is the lexicographically first member
of ${\cal S}$. Concretely, write each pair under consideration as
$\{s_i,s_j\}$ with $i>j$. Now sort the pairs in each $S \in {\cal
S}$ in increasing lexicographic order and let $T \in {\cal S}$ be
the lexicographically first member of ${\cal S}$. Our claim is
that $T=\widehat{S}$. Observe that if $\{s_i,s_j\} \in T$ then
there is no index $h < i$ with $h \neq j$ with $\{s_i,s_h\} \in T$.
Otherwise we could replace the pair $\{s_i,s_j\}$ with the pair
$\{s_h,s_j\}$ and generate the same quotient as does $T$ with a
lexicographically smaller set of pairs.

It is helpful to consider for each $1 \leq i \leq |w|$ the graph
$\Gamma_i$ that is the quotient of $s_0 \to \ldots \to s_i$
generated by the initial segments of $T$ that includes only those
pairs in $T$ where both indices are $\le i$. We claim that
$\Gamma_{i-1}$ is a subgraph of $\Gamma_i$ for all $i$. If there
is no pair in $T$ where $s_i$ is the larger member, this is clear.
If $\{s_i,s_j\}$ is in $T$ then the index $j$ is uniquely defined
by the above remark. In this case we need to show that the
identification of $s_i$ with $s_j$ does not entail any additional
identification (which would violate the inclusion $\Gamma_{i-1}
\subseteq \Gamma_i$). How can such an identification occur?
Only if the label of the edge $(s_{i-1},s_i)$ agrees
with that of an edge $e$ incident with the block that contains
$s_j$ (and has the correct orientation). Let $s_{\nu}$ be a member
of the block at the other end of $e$. Clearly $\nu < i$. But now
again we can generate the quotient generated by $T$ by a
lexicographically smaller set, i.e., replace the pair
$\{s_i,s_j\}$ by $\{s_{i-1},s_{\nu}\}$.

We can again recognize the three types of steps by observing how
the graphs $\Gamma_i$ grow at each step. If $\Gamma_i$ stays
unchanged, this is a forced move. If only an edge is added, this
is a coincidence and in a free step one vertex and one edge are
added. It is exactly at each coincidence step that vertices at $T$
get merged. But this is precisely what we did in constructing
$\widehat{S}$, so that $\widehat{S} = T$, as claimed.

\end{proof}

The following categorization of the quotient graphs in ${\cal
Q}_w$ turns out to be very useful:

\begin{defi}\label{df:AB}
Let $w$ be a word in $\Sigma_k$. We say that a quotient graph
$\Gamma \in {\cal Q}_w$ \textbf{has type A}, if one of the
smallest generating sets for $\Gamma$ contains the pair
$\{s_0,s_{|w|}\}$. Otherwise, we say $\Gamma$ \textbf{has type B}.
\end{defi}

Given a word $w$, we classify the graphs in ${\cal Q}_w$ according
to their characteristics and type. Note that $\chi(\Gamma) \leq
|w|$, since every $\Gamma$ has at most $|w|$ edges. We illustrate
this again with the seven graphs of the commutator word: The
figure-eight graph with one vertex and two edges has type B. The
other six graphs have type A (their generating sets specified in
Figure \ref{fig:minimal-gps} purposely include $\{s_0,s_4\}$).
Table \ref{tab:chi-type} shows the whole census.

\begin{table}[htb]
\centering
\begin{tabular}[c]{|c|c|c|c|c|c|}
\hline $\chi$ & 0 & 1 & 2 & 3 & 4\\
\hline \hline Type A & 0 & 1 & 4 & 1 & 0\\
\hline Type B & 0 & 0 & 1 & 0 & 0\\
\hline
\end{tabular} \caption{The number of graphs in ${\cal Q}_w$
in any specified character and type, when
$w=g_1g_2g_1^{\;-1}g_2^{\;-1}$.} \label{tab:chi-type}
\end{table}

We are now ready for the second definition for a word's ``level of
primitivity''. Let $w$ be a word in $\Sigma_k$. We define
$\beta(w)$ to be the smallest characteristic of a type-B graph in
${\cal Q}_w$. Namely,
\begin{equation}\label{eq:def-beta}
\mathbf{\beta(w)} := \min\left( \{ \chi(\Gamma)~:~ \Gamma \in {\cal
Q}_w \textrm{ and $\Gamma$ has type B} \} \cup \{\infty\} \right)
\end{equation}

\begin{example}
For the commutator word $\beta(w)=2$.
\end{example}

In the next few lemmas we establish several properties of
$\beta(\cdot)$ which are clearly desirable. Among others,
$\beta(\cdot)$ is proved to be invariant under reductions, and
hence well defined as a function on $\mathbf{F}_k$.

\begin{lem} \label{lem:beta-invrt-cyc-shft}
$\beta(\cdot)$ is invariant under cyclic shifts.
\end{lem}
\begin{proof}
Let $w \in \Sigma_k$, and let $w'$ be some cyclic shift of $w$.
There is an obvious bijection between ${\cal Q}_w$ and ${\cal
Q}_{w'}$, obtained by applying the appropriate cyclic shift on the
indices of the $s_i$'s in the blocks' names. (E.g, if $w'$ is
attained from $w$ by a right cyclic shift of two positions, replace
each label $s_i$ in $w$ by $s'_j$ in $w'$ where
$j \equiv i+2 \mod |w|$.) We claim that in addition, each $\Gamma \in
{\cal Q}_w$ has the same type in ${\cal Q}_w$ as its matching
quotient $\Gamma'$ in ${\cal Q}_{w'}$. It suffices to show that if
$\Gamma$ has type A in ${\cal Q}_w$, then $\Gamma'$ have type A in
${\cal Q}_{w'}$ (It then follows by symmetry that $\Gamma$ has type A
in ${\cal Q}_w$ iff $\Gamma'$ has type A in ${\cal Q}_{w'}$).

Let $S$ be a smallest generating set of $\Gamma$ with
$\{s_0,s_{|w|}\} \in S$. Given $S$, we can generate $\Gamma$
gradually, through a series of quotients. To proceed from the
quotient $\Gamma_i$ to the next quotient, we add a pair $\{s_j,
s_r\} \in S$. To determine $\Gamma_{i+1}$ we carry out all
necessary identifications and only them. Formally, $\Gamma_{i+1}$
is the finest realizable quotient of $\Gamma_i$ in which the pair $\{s_j,
s_r\}$ is merged. It is easily verified that the final quotient is
$\Gamma$ regardless of the order at which the pairs in $S$ are
introduced. But merging $\{s_0,s_{|w|}\}$ in $w$, and merging
$\{s'_0,s'_{|w|}\}$ in $w'$, yield equivalent quotient graphs
(these are the universal graphs of $w$ and of $w'$, as defined in
Section \ref{sbs:exp-fix-points} around Figure
\ref{fig:univ-graphs}). We can now proceed by applying the same
series of quotients as described above on both universal graphs,
to conclude that $\Gamma'$ is of type A with respect to $w'$ as
well.
\end{proof}

\begin{lem}\label{lem:beta-invrt-red}
Let $w \in \Sigma_k$, and
let $w'$ be its reduced form. Then $\beta(w) = \beta(w')$.
\end{lem}

\begin{proof}
We need to show that $\beta$ does not change when a letter and its
inverse are inserted consecutively into a word. But Lemma
\ref{lem:beta-invrt-cyc-shft} says that $\beta$ is invariant under
cyclic shifts, so it suffices to show that $\beta(w)=\beta(w')$
for $w=g_{i_1}^{\;\alpha_1}g_{i_2}^{\;\alpha_2}\ldots
g_{i_m}^{\;\alpha_m}$ and
$w'=g_{i_1}^{\;\alpha_1}g_{i_2}^{\;\alpha_2}\ldots
g_{i_m}^{\;\alpha_m}g_j g_j^{\;-1}$. To see this, we define below
for every quotient $\Gamma \in {\cal Q}_{w}$ the subset
$\epsilon(\Gamma) \subseteq {\cal Q}_{w'}$ of all consistent
extensions of $\Gamma$. The set $\epsilon(\Gamma)$ contains a
certain member $\delta(\Gamma)$ which plays a special role. The
relevant properties of $\epsilon$ and $\delta$ are:
\begin{itemize}
\item
The union of the images of $\epsilon$ is all of ${\cal Q}_{w'}$.
\item
If $\Gamma_1 \neq \Gamma_2 \in {\cal Q}_{w}$, then
$\epsilon(\Gamma_1)$ and $\epsilon(\Gamma_2)$ are disjoint.
\item
If $\Gamma \in {\cal Q}_{w}$ has type A, then all members in
$\epsilon(\Gamma)$ have type A.
\item
For every $\Gamma \in {\cal Q}_{w}$, one graph $\delta(\Gamma) \in
\epsilon(\Gamma)$ has the same Euler characteristic and the same
type as $\Gamma$. All other members in $\epsilon(\Gamma)$ have
Euler characteristic $\chi(\Gamma)+1$.
\end{itemize}
It should be clear that these properties prove the lemma.

If $v \in V(\Gamma)$ is the vertex corresponding to $s_m$, then
$\epsilon(\Gamma)$ is the set of all extensions of $\Gamma \in
{\cal Q}_w$ where there is a $j$-edge (an edge labeled $j$)
starting at $v$. If $\Gamma$ already has such an edge, then we can
attain a graph $\Gamma' \in {\cal Q}_{w'}$ by adding $s_{m+2}$ to
the block containing $s_m$, and adding $s_{m+1}$ to the block at
the end of this $j$-edge. We then define $\epsilon(\Gamma) = \{
\delta(\Gamma) \} = \{ \Gamma'\} \approx \{\Gamma\}$ (We use
``$\approx$'' to denote equality as vertex-unlabeled-graphs.)

Otherwise, $\epsilon(\Gamma)$ includes all the
$(v_\Gamma-e_\Gamma^j+1)$ different possible extensions of
$\Gamma$ with such an edge. We only need to specify the other
vertex of this new edge, that corresponds to $s_{m+1}$. In the
graph $\delta(\Gamma)$ the vertex $s_{m+1}$ is new and so is the
$j$-edge $(s_m, s_{m+1})$.
Clearly, $\chi(\delta(\Gamma))=
\chi(\Gamma)$, as claimed. Otherwise this additional edge can go
from $s_m$ to any of the $v_\Gamma-e_\Gamma^j$ vertices in
$\Gamma$ which are not tails of a $j$-edge. Such graphs clearly
have characteristics $\chi(\Gamma)+1$.

We prove the first two properties of $\epsilon$ by recovering, for
every $\Gamma' \in {\cal Q}_{w'}$ the (unique) graph $\Gamma \in
{\cal Q}_{w}$ with $\Gamma' \in \epsilon(\Gamma)$. We consider the
$(m+1)$-st step in the path of $w'$ through $\Gamma'$ (the step
from $s_m$ to $s_{m+1}$), and use the notations of Lemma
\ref{lem:char=mgps}. If this step is free, then $\Gamma$ is
obtained from $\Gamma'$ by deleting the vertex corresponding to
$s_{m+1}$ and the edge $(s_m,s_{m+1})$ (as well as, of course,
omitting $s_{m+2}$ from its block). If it is a coincidence, then
clearly $\Gamma \approx \Gamma' \setminus (s_m, s_{m+1})$.
Otherwise, it is forced and so $\Gamma \approx \Gamma'$.

We want to show next that if $\Gamma \in {\cal Q}_w$ has type A,
then all graphs in $\epsilon(\Gamma)$ have type A as well. By
assumption $\Gamma$ is generated by a set $S$ of cardinality
$|S|=\chi(\Gamma)$ and $\{s_0,s_m\} \in S$. Note that $s_m \equiv
s_{m+2}$ in \emph{every} quotient of $w'$. Therefore, when we
consider generating sets for graphs in ${\cal Q}_{w'}$, the
vertices $s_m$ and $s_{m+2}$ play the exact same role. We
therefore define $S'$ to be the set of pairs that is attained by
replacing each occurrence of $s_m$ in $S$ with $s_{m+2}$. Clearly,
$S'$ generates the graph $\delta(\Gamma) \in \epsilon(\Gamma)$.
For any other graph in $\epsilon(\Gamma)$, we add to $S'$ the pair
$\{s_{m+1},s_i\}$ where $s_i~(i \leq m)$ corresponds to the vertex
which $(s_m,s_{m+1})$ goes to. In each of these cases we found a
smallest generating set $S'$ that includes the pair
$\{s_0,s_{m+2}\}$, so all members of $\epsilon(\Gamma)$ have type
A.

Finally, we need to show that if $\Gamma \in {\cal Q}_w$ has type
B, then so does $\delta(\Gamma)$. So suppose $\delta(\Gamma)$ has
type A, with a smallest generating set $S'$ that contains the pair
$\{s_0,s_{m+2}\}$. Let us construct a set of pairs $S$ by
replacing each occurrence of $s_{m+2}$ in $S'$ by $s_m$. If $S'$
contains some pair $\{s_{m+1},s_i\}$, then clearly $s_{m+1}$ is
not a new vertex in $\delta(\Gamma)$, and we are necessarily in
the case where $\Gamma \approx \delta(\Gamma)$ (recall that
$\Gamma$ and $\delta(\Gamma)$ have the same characteristic). In
this case the edge $(s_m,s_{m+1})$ is not new, so it is merged
with some $(s_{r-1},s_r)$ (or $(s_{r+1},s_r)$) for some $r<m$.
Thus, we can replace each $s_{m+1}$ in $S'$ with $s_r$. This is a
contradiction since $S$ is a smallest generating set of $\Gamma$
which therefore has type A.
\end{proof}

Note that from Lemmas \ref{lem:beta-invrt-cyc-shft} and
\ref{lem:beta-invrt-red} it follows that $\beta$ is invariant
under \emph{cyclic} reduction as well, or under conjugation.
Similar arguments show that it is also invariant under the
equivalence relation $``\sim"$, but we do not include the proof.
In the following lemma we state an important property of type-B
quotient graphs. This property plays a crucial role in the sequel,
where we introduce a bound to the number of words with some fixed
value of $\beta(\cdot)$.

\begin{lem}\label{lem:2-preimages}
Let $\Gamma \in {\cal Q}_w$ have type B. As we trace the path of
$w$ through $\Gamma$, every edge in $\Gamma$ is traversed at least
twice.
\end{lem}
\begin{proof}
We show that if some edge $e$ is traversed only once, then
$\Gamma$ has type A. Lemma \ref{lem:beta-invrt-cyc-shft} allows us
to assume that $e$ is the last step in the path of $w$, i.e. the
step from $s_{|w|-1}$ to $s_{|w|}$. In the proof of Lemma
\ref{lem:char=mgps}, we constructed $\widehat{S}$, a generating
set of $\Gamma$ of smallest cardinality, with one pair for each
coincidence in the path of $w$ through $\Gamma$. Here, the last
coincidence corresponds to the pair $\{s_0,s_{|w|}\}$. Therefore
$\Gamma$ has type A.
\end{proof}

\begin{remark}
The converse is not true. There are quotients of type A where
every edge is traversed more than once. Consider the word
$w=ababa$. One of the quotient graphs in ${\cal Q}_w$ is a
figure-eight with one vertex and two loops. Each edge in this
quotient is traversed twice or thrice, but the quotient has type
A. It is generated by the two pairs $ \{s_0,s_2\},\{s_0,s_5\}$.
\end{remark}

\subsection{Some Connections between $\phi(w)$ and $\beta(w)$}
\label{sbs:phi=beta}

We now turn to examine the relation between $\phi$ and $\beta$. We
first observe that $a_i(w)$ (the coefficient of $\frac{1}{n^i}$ in
the power series form of $\Phi_w(n)$), is completely determined by
quotients in ${\cal Q}_w$ with characteristic $\leq i$. This is
easily verified by considering the contribution of each $\Gamma
\in {\cal Q}_w$ in \eqref{eq:Phi-series}. We are now able to use
our new perspective and show that (as mentioned above) for
$i=0,1$, $\phi(w)=i \Leftrightarrow \beta(w)=i$.

\begin{lem} \label{lem:phi=beta=0}
For every $w \in \mathbf{F}_k$, $\phi(w)=0 \Leftrightarrow
\beta(w)=0 \Leftrightarrow w=1$.
\end{lem}
\begin{proof}
By Lemma \ref{lem:char=mgps}, the only quotient graph of $w$ with
characteristic 0 is the graph $\Gamma$ generated by the empty set.
Now $\phi(w)=0 ~ \Leftrightarrow ~ a_0(w)>0 ~ \Leftrightarrow ~
\Gamma \in {\cal Q}_w ~ \Leftrightarrow s_0 \equiv s_{|w|} \textrm{
in } \Gamma \Leftrightarrow w$ reduces to 1. If $\Gamma \in {\cal
Q}_w$ then $\Gamma$ has type B by definition. Thus $\beta(w)=0
\Leftrightarrow \Gamma \in {\cal Q}_w$.
\end{proof}

\begin{lem} \label{lem:phi=beta=1}
For every $w \in \mathbf{F}_k$, $\phi(w)=1 \Leftrightarrow
\beta(w)=1 \Leftrightarrow$ $w$ is imprimitive.
\end{lem}
\begin{proof}
Let $w\in \mathbf{F}_k$ be in reduced form and assume $w \neq 1$.
By Lemma \ref{lem:phi=beta=0}, all $\Gamma \in {\cal Q}_w$ have a
positive characteristic. The definition of $\Phi_w(n)$ clearly
yields that $a_1(w) = \left| \{\Gamma \in {\cal
Q}_w~:~\chi(\Gamma)=1\}\right| - 1$. In this case the single pair
$\{s_0,s_{|w|}\}$ is a smallest generating set for the universal
graph (defined in Section \ref{sbs:exp-fix-points}), so at least
one quotient has characteristic 1. Obviously, any other quotient
with $\chi=1$ is not generated by $\{s_0,s_{|w|}\}$, and thus (by
Lemma \ref{lem:char=mgps}) has
type B. Thus $\phi(w)=1 \Leftrightarrow \beta(w)=1
\Leftrightarrow$ such additional quotients exist. We complete the
proof by showing that the latter is true iff $w$ is imprimitive.

Let $w'$ be the cyclic reduction of $w$. It is easy to verify that
$w \sim w'$ whence $\phi(w)=\phi(w')$ and $w$ is primitive iff so
is $w'$. Lemmas \ref{lem:beta-invrt-cyc-shft} and
\ref{lem:beta-invrt-red} yield that $\beta(w)=\beta(w')$ as well.
Thus we can assume, for simplicity, that $w$ is cyclically
reduced.

In this case, merging $s_0$ and $s_{|w|}$ implies no other
identifications, and the universal graph is a cycle of length
$|w|$. If $w$ is imprimitive, there is some $u \in \mathbf{F}_k$
and $d\geq 2$ such that $w=u^d$. Clearly, $u$ is cyclically
reduced as well, and the universal graph of $u$ is a cycle of
length $|u|$ which is an additional quotient of characteristic 1
in ${\cal Q}_w$.

On the other hand, since $w$ is cyclically reduced, every vertex
in every $\Gamma \in {\cal Q}_w$ has degree $\geq 2$. Thus, if
$\chi(\Gamma)=1$, it is necessarily a cycle. As the path of $w$
through $\Gamma$ is non-backtracking and $w\neq1$, it consists of
tracing this cycle some $d\geq 1$ times. If $d=1$, $\Gamma$ is the
universal graph. Otherwise, if we let $u$ denote the word
corresponding to a single traversal of the cycle, then $w=u^d$.
\end{proof}

The contents of Lemmas \ref{lem:phi=beta=0} and
\ref{lem:phi=beta=1} appear in different language in \cite{bs},
in \cite{nica} and in \cite{fri03}. But the relation between
$\phi(\cdot)$ and $\beta(\cdot)$ goes deeper. For instance, for
the single-letter word $w=a$ both $\phi(w) = \beta(w) = \infty$.
The reason for $\phi(a)=\infty$ is that the expected number of
fixed points in a random permutation equals 1. On the other hand,
Lemma \ref{lem:2-preimages} implies that $\beta(a)=\infty$. Also,
as already mentioned, both $\phi$ and $\beta$ are invariant under
$``\sim"$, so they are both infinite on the entire equivalence
class of $a$ under $``\sim"$. (In particular, every $w$ in which
some letter appears exactly once belongs to this class.)

The following lemma expands even further the relation between
$\phi(\cdot)$ and $\beta(\cdot)$. The next natural step would be
to prove that $\phi(w)=2 \Leftrightarrow \beta(w)=2$. This, in
other words, says that for primitive words, $\beta(w) \geq 3$ iff
$a_2(w) = 0$. This is, at present, still beyond our reach and we
content ourselves with a weaker statement.

\begin{lem} \label{lem:phi approx beta=2}
Let $w \in \mathbf{F}_k$ have $\beta(w) \geq 3$. Then $a_2(w) \leq 0$.
\end{lem}
\begin{proof}
For simplicity we assume that $w$ is cyclically reduced. (Again,
this assumption is possible because both $\beta(w)$ and
$\Phi_w(n)$ are invariant under cyclic reductions of $w$.) When
$\beta(w) \geq 3$, there is only one graph $\hat{\Gamma} \in {\cal
Q}_w$ of characteristic 1 (the universal graph), and all the
quotient graphs of characteristic 2 have type A. To find the
contribution of $\hat{\Gamma}$ to $a_2(w)$ expand the expression
$\frac{\prod_{l=1}^{v_\Gamma-1}(1-lx)}{\prod_{j=1}^k
\prod_{l=1}^{e_\Gamma^j-1} (1-lx)}$ to first order. It follows
that this contribution is $-\binom{v_{\hat{\Gamma}}}{2} +
\sum_{j=1}^k \binom {e^j_{\hat{\Gamma}}}{2}$. We need to show that
there are at most $\binom{v_{\hat{\Gamma}}}{2} - \sum_{j=1}^k
\binom {e^j_{\hat{\Gamma}}}{2}$ graphs $\Gamma \in {\cal Q}_w$
with $\chi(\Gamma)=2$.

Since $\hat{\Gamma}$ is generated by $\{s_0,s_{|w|}\}$, and every
quotient $\Gamma \in {\cal Q}_w$ of characteristic 2 has type A,
$\Gamma$ is generated from $\hat{\Gamma}$ by a single pair of
vertices of $\hat{\Gamma}$. The total number of pairs is
$\binom{v_{\hat{\Gamma}}}{2}$, but clearly different pairs may
generate the same $\Gamma$. For instance, for any two $j$-edges, the
pair of heads generates the same quotient as the pair of tails.

In order to understand the full picture, we introduce a graph
$\Upsilon$, which captures this kind of dependency between pairs.
The graph $\Upsilon$ has $\binom{v_{\hat{\Gamma}}}{2}$ vertices
labeled by the pairs of vertices of $\hat{\Gamma}$, and has
$\sum_{j=1}^k \binom {e^j_{\hat{\Gamma}}}{2}$ edges, one for each
pair of same-color edges in $\hat{\Gamma}$. The edge corresponding
to the pair $\{\epsilon_1,\epsilon_2\}$ of $j$-edges, is a
$j$-edge from the vertex $\{head(\epsilon_1) ,head(\epsilon_2)\}$
to $\{tail(\epsilon_1) ,tail(\epsilon_2)\}$. We illustrate this in
Figure \ref{fig:Upsilon}.

\begin{figure}[htb]
\centering
\includegraphics[width=1\textwidth]{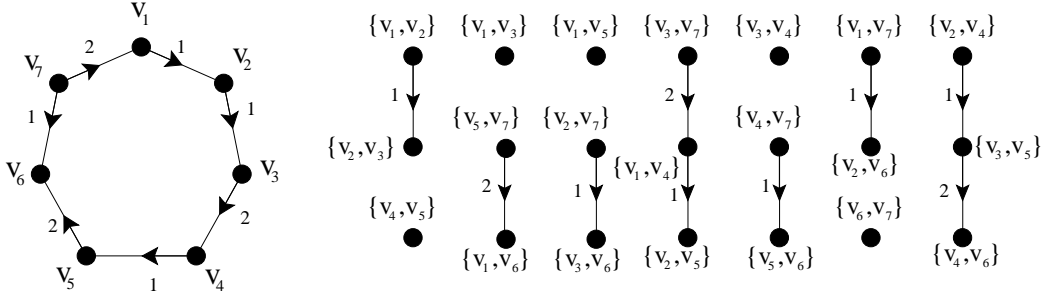}
\caption{The graph $\Upsilon$ (on the right) corresponding to the
universal graph $\hat{\Gamma}$ (on the left) for
$w=g_1^2 g_2 g_1 g_2 g_1^{\;-1} g_2$. ($\hat{\Gamma}$'s
vertices are denoted here by $v_1,\ldots,v_7$ while the $s_i$ labels
are omitted.) }  \label{fig:Upsilon}
\end{figure}

We claim that $\Upsilon$ has no cycles. As $w$ is assumed to be
cyclically reduced, $\hat{\Gamma}$ is simply a cycle and the path
of $w$ through it is a simple cycle. Now say that
$\{x_1,y_1\},\{x_2,y_2\},\ldots,\{x_r,y_r\},\{x_1,y_1\}$ is a
simple cycle in $\Upsilon$, whose edges compose some word $u$.
This $u$ corresponds to two non-backtracking paths in
$\hat{\Gamma}$ in one of the two following ways. Either $u$ is a
path from $x_1$ to itself and a path from $y_1$ to itself (whence it
is some cyclic shift of $w$ or of $w^{-1}$). Or $u$
is a path from $x_1$ to $y_1$ as well as a path from $y_1$ to
$x_1$. (See Figure \ref{fig:cycle_in_upsilon}.)

\begin{figure}[htb]
\centering
\includegraphics[width=1\textwidth]{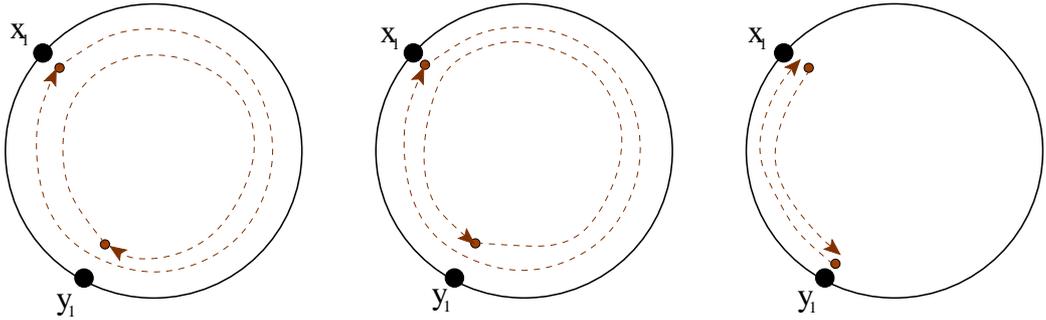}
\caption{A cycle in $\Upsilon$ yields one of these three scenarios. The
thick line stands for the universal graph $\hat{\Gamma}$. The broken
lines (arrows) stand for copies of $u$, the subword composed by the
edges of the cycle in $\Upsilon$. In the left and middle scenarios,
$u$ is a path from $x_1$ to itself and a path from $y_1$ to itself (in
the left graph the paths are equally oriented, in the middle one they
have inconsistent orientation). The right graph describes the case where
$u$ is a path from $x_1$ to $y_1$ as well as a path from $y_1$ to
$x_1$. All three scenarios are impossible.}  \label{fig:cycle_in_upsilon}
\end{figure}

In the former case, some cyclic shift of $w$ equals another
cyclic shift of $w$ or of $w^{-1}$. But $w$ is primitive, so
it is not invariant under any cyclic shift. To see that $w$ cannot
equal a cyclic shift of its inverse, we need to show that $w^{-1}$
is not a subword of $w^2$. To see this, let $w=g_1\ldots g_m$ and
assume to the contrary that $\forall j=1,\ldots,m ~~ g_j^{\;
-1}=g_{s-j}$ for some $m+1 \leq s \leq 2m$ (here $g_{i+m}=g_i$ for
$i=1,\ldots,m$). Is $s$ is even, say $s=2q$, we conclude (for
$j=q$) that $g_q=g_q^{\; -1}$ which is impossible. If $s=2q+1$ we
conclude for $j=q$ that $g_q^{\; -1} = g_{q+1}$, so that the word
$w$ is not reduced.

In the latter case, either $u=u^{-1}$, which is impossible by the
same argument, or $w$ is a cyclic shift of $u^2$, which again
contradicts its being primitive. This rules out the possibility of
a cycle in $\Upsilon$.

Clearly, two pairs of vertices in $V(\hat{\Gamma})$ which belong to
the same connected component in $\Upsilon$, generate the same
$\Gamma$. Thus the number of different $\Gamma$'s in ${\cal Q}_w$
with characteristic 2 is at most the number of connected components
in $\Upsilon$. But $\Upsilon$ has no cycles, so the number of
connected components is exactly $\binom{v_{\hat{\Gamma}}}{2} -
\sum_{j=1}^k \binom {e^j_{\hat{\Gamma}}}{2}$.
\end{proof}

Note that this last proof yields a surjective function from the
connected components of $\Upsilon$ to type-A graphs of
characteristics 2 in ${\cal Q}_w$. In order to prove that
$\phi(w)=2 \Leftrightarrow \beta(w)=2$ it suffices to show that
this function is injective (which we believe is true). That would
yield that the number of type-A quotients of characteristic 2
exactly balances off the negative contribution of $\hat{\Gamma}$
to $a_2(w)$. Thus, $a_2(w) \neq 0$ (or more precisely $a_2(w)>0$)
iff there is a type-B graph of characteristic 2, i.e. iff
$\beta(w)=2$.

In fact, we believe that something similar happens in general.
Namely, for every integer $i\geq 0$, if ${\cal Q}_w$ has no
type-B graphs of characteristic $< i$, then the
contributions from all the type-A graphs of characteristic $\leq
i$. (i.e., the graphs generated from the universal graph by fewer
than $i$ pairs) balances out. Hence
$a_0(w)=a_1(w)=\ldots=a_{i-1}(w)=0$, and $a_i(w)$ equals the
number of type-B quotients of characteristic $i$.

There are three kinds of supporting evidence to this belief. As we
saw, it is valid for $i=0,1$, and we have a good understanding
why it should hold for $i=2$ as well. Also, we have carried out extensive
numerical simulations to test it for $i=2,3,4$ without any failure.
Finally, consider a word $w$ such that $w \sim a$. Such a word has
only type-A quotients, and we know that $a_i(w) = 0$ for all $i$.
In this case, therefore, the type-A quotients of characteristic
$\leq i$ indeed balance the contribution of each other to $a_0(w),
\ldots, a_i(w)$. (We note that the vanishing of all $a_i(w)$ even
in this specialized case is not obvious).

Here, then, is a formal statement of this main conjecture:
\begin{conj} \label{conj:beta=phi}
For every $w \in \mathbf{F}_k$,
\begin{equation}\label{eq:phi=beta}
\phi(w) = \beta(w)
\end{equation}

Moreover, if $\phi(w)=i$, then
\begin{equation}\label{eq:a_i-equals...}
a_i(w) = \left| \{\Gamma \in {\cal Q}_w~:~\chi(\Gamma)=i \textrm{
and $\Gamma$ has type B}\} \right|
\end{equation}
\end{conj}

Conjecture \ref{conj:beta=phi} is the main missing link in our
tentative proof to the claim that $\mathbb{E}(\mu_{max}) <
O(\rho)$. We return to this in Section \ref{sec:lifts}. Note that
this conjecture implies that the first non-zero coefficient in the
power series of $\Phi_w(n)$ is positive. This means that
$E(X_w^{(n)}) \geq 1$ for every $w \in \mathbf{F}_k$ and every
large enough $n$.

\begin{remark}
It might be tempting to suspect something even stronger, namely that
for every $w \in \mathbf{F}_k$ and $n \geq 1$, $E(X_w^{(n)}) \geq 1$.
However, this stronger assertion is false. (We would like to thank
Mikl\'{o}s Ab\'{e}rt for showing us the invalidity of this claim. The main
ideas of the proof can be found in \cite{abert}.)
\end{remark}

Finally, we present a second conjecture based on other simulations
we conducted. These simulations suggest that the only words for
which $\phi(w)=\infty$ are those mentioned above:

\begin{conj}
\label{nielsen_conj}
$E(X_w^{(n)}) \equiv 1$ iff $w$ is equivalent ($\sim$) to a
single-letter word.
\end{conj}


\section{The Largest New Eigenvalue in a Random Lift of a Graph}
\label{sec:lifts}

In this section we apply our findings concerning formal words and word
maps on $S_n$ to study the new eigenvalues in random lifts
of graphs. Our main result is
Theorem \ref{ther:2/3} which says that $\mu_{max} \leq
O(\lambda_1^{\;1/3}\rho^{2/3})$ almost surely.

Recall the definition of $\phi(w)$ for $w\in \mathbf{F}_k$
(Equation \eqref{eq:phi}). Namely, $\Phi_w(n) =
\frac{a_{\phi(w)}(w)+o(1)}{n^{\phi(w)}}$. We first seek an
improved estimate of the $o(1)$ term in the numerator.

\begin{lem} \label{lem:n_manot}
Let $w\in \Sigma_k$ and let $i\geq 0$ be some integer. Then:
\[\left|\{\Gamma \in {\cal Q}_w~:~\chi(\Gamma)=i \}\right| \leq |w|^{2i}\]
\end{lem}

\begin{proof}
As we saw in Lemma \ref{lem:char=mgps}, each $\Gamma \in {\cal Q}_w$
with characteristic $i$ is generated by some set of $i$ pairs. There
are $\binom{|w| + 1}{2} \leq |w|^2$ pairs to choose from, and the
claim follows.
\end{proof}

\begin{lem} \label{lem:analyze-o(1)}
If $\phi(w) = i$ and if $n \geq 3|w|^2$, then
\begin{equation} \label{eq:analyze-o(1)}
\Phi_w(n) \leq
\frac{1}{n^i}\left(a_i(w)+\frac{|w|^{2i+4}}{n}\right).
\end{equation}
\end{lem}

\begin{proof}
As mentioned in the beginning of Section \ref{sbs:phi=beta},
$a_i(w)$, the coefficient of $\frac{1}{n^i}$ in the power series
of $\Phi_w(n)$, is completely determined by quotients in ${\cal
Q}_w$ of characteristic $\leq i$. We now bound the contribution to
$a_i(w)$ of every such quotient.

To this end, we analyze the contribution of $\Gamma$ to
$\Phi_w(n)$, as specified in \eqref{eq:Phi-series}. For the sake
of convenience, we let $x$ equal $\frac{1}{n}$, and express this
contribution as $x^{e_\Gamma-v_\Gamma+1} \cdot
\frac{\prod_{l=1}^{v_\Gamma-1}(1-lx)}{\prod_{j=1}^k
\prod_{l=1}^{e_\Gamma^j-1} (1-lx)}$. For small $x$ we can expand
the fraction in this expression as a power series
$\sum_{r=0}^\infty b_r x^r$. Write $\prod_{l=1}^{v_\Gamma-1}(1-lx)
= 1+ \sum_{r \ge 1} c_r x^r$, and $\prod_{j=1}^k
\prod_{l=1}^{e_\Gamma^j - 1} (1-lx) = 1 + \sum_{r \ge 1} d_r x^r$.
Then:
\[ 1 + \sum_{r=1}^{\infty}c_rx^r = \left[ \sum_{r=0}^\infty b_r x^r
\right] \left[ 1+ \sum_{r=1}^{\infty} d_r x^r \right] \] Thus
$b_0=1$ and for $r\geq 1$, $b_r = c_r - d_r - b_1d_{r-1} - ...
-b_{r-1}d_1$. We have
\begin{displaymath}
|c_r| = \sum_{1 \leq y_1 < \ldots < y_r \leq v_\Gamma-1} y_1 \ldots
y_r \leq \left[\sum_{y=1}^{v_\Gamma-1} y\right]^r \leq
\left(\frac{v_\Gamma^{\;2}}{2}\right)^r \leq \frac{|w|^{2r}}{2^r}
\end{displaymath}
Similarly, $|d_r| \leq \frac{|w|^{2r}}{2^r}$. A simple induction now
shows that $|b_r| \leq |w|^{2r}$:
\begin{eqnarray*}
|b_r| &=& |c_r - d_r - b_1d_{r-1} - ... -b_{r-1}d_1| \leq \\
&\leq& \frac{|w|^{2r}}{2^r} + \frac{|w|^{2r}}{2^r} +
\frac{|w|^2|w|^{2r-2}}{2^{r-1}} + \frac{|w|^4|w|^{2r-4}}{2^{r-2}} +
\ldots + \frac{|w|^{2r-2}|w|^2}{2} = |w|^{2r}
\end{eqnarray*}

Now consider the coefficient $a_i(w)$. There is at most one quotient
in ${\cal Q}_w$ of characteristic 0 (Lemma \ref{lem:n_manot}) which
contributes at most $|w|^{2i}$ to $a_i(w)$; There are at most
$|w|^2$ quotients of characteristic 1 which contribute at most
$|w|^{2i-2}$ each, etc. There are no quotients with characteristic
greater then $|w|$, so we have at most contribution of $|w|^{2i}$ of
quotients of every characteristic $0 \leq \chi \leq |w|$. Thus
$a_i(w) \leq (|w|+1)|w|^{2i}$.

This yields the following:
\begin{eqnarray*}
\Phi_w(n) &=& \sum_{i=\phi(w)}^\infty a_i(w)\frac{1}{n^i} \leq
a_{\phi(w)}(w) \frac{1}{n^{\phi(w)}} + \sum_{i=\phi(w)+1}^\infty
\frac {(|w|+1)|w|^{2i}}{n^i} = \\ &=& \frac{1}{n^{\phi(w)}} \left[
a_{\phi(w)}(w) + \frac{(|w|+1)|w|^{2\phi(w)+2}}{n} \frac{n}{n-|w|^2}
\right]
\end{eqnarray*}
The lemma now follows because $n \geq 3|w|^2$.
\end{proof}

The set of possible values for $\beta(w)$ is
$\{0,1,\ldots,|w|\}\cup\{\infty\}$, and we now
split the sum over $w \in {\cal CP}_t(G)$
in \eqref{eq:trace} according to $\beta(w)$. This yields:
\begin{eqnarray*}
\mathbb{E}(\mu_{max}^{\;~~~t}) &\leq& \sum_{w \in {\cal
CP}_t(G)} \left[\mathbb{E}(X_w^{(n)}) - 1
\right] = \sum_{w \in {\cal CP}_t(G)} n\cdot \Phi_w(n) \\
&=& \sum_{i \in \{0,1,\ldots,t\}\cup\{\infty\}} \sum_{\substack{w
\in {\cal CP}_t(G) \\ \beta(w)=i}} n\cdot \Phi_w(n)
\end{eqnarray*}

The statement of (the unproved) Conjecture \ref{conj:beta=phi}
implies that the sum over $w$ with $\beta(w) = \infty$ vanishes,
since $\beta(w) = \infty$ yields $\phi(w) = \infty$ and hence
$\Phi_w(n) \equiv 0$. We suspect that it should be possible to
bound the number of words with $\beta(w)=i$ (Some results along
these lines are proved in Section \ref{sbs:n_words_with_beta=}). This, combined
with the statement of Conjecture \ref{conj:beta=phi}
would have allowed us to
bound the contribution of each $0\leq i \leq t$ to the above sum.

This problem is still open, so instead we split the set ${\cal
CP}_t(G)$ into four parts:
\begin{eqnarray} \label{eq:split-by-beta}
\mathbb{E}(\mu_{max}^{\;~~~t}) \leq & \sum_{\substack{w \in
{\cal CP}_t(G)\\ \beta(w)=0}} n \cdot \Phi_w(n) +
\sum_{\substack{w \in {\cal CP}_t(G)\\ \beta(w)=1}} n \cdot \Phi_w(n) + \nonumber \\
& \sum_{\substack{w \in {\cal CP}_t(G)\\ \beta(w)=2}} n \cdot
\Phi_w(n) + \sum_{\substack{w \in {\cal CP}_t(G)\\ \beta(w)\geq 3}}
n\cdot \Phi_w(n) \label{beta-split}
\end{eqnarray}

Using Lemmas \ref{lem:phi=beta=0} and \ref{lem:phi=beta=1} we can
bound the value of $\Phi_w(n)$ when $\beta(w)=0$ or $1$.
For these values $\beta(w)= \phi(w)$ and
Lemma~\ref{lem:analyze-o(1)} can be applied. If $\beta(w)=2$, then
$\phi(w) \geq 2$, and we can use \eqref{eq:analyze-o(1)} in its
worst case, i.e. when $\phi(w)=2$. Finally, if $\beta(w)\geq 3$,
the following lemma shows that $\Phi_w(n)\leq O\left(
\frac{1}{n^3} \right)$.

\begin{lem} \label{lem:beta>=3}
Let $w \in \mathbf{F}_k$ have $\beta(w) \geq 3$. Then \[\Phi_w(n)
\leq \frac{1}{n^3} \left( a_3(w) + \frac{|w|^{10}}{n} \right)\]
\end{lem}
\begin{proof}
The assumption $\beta(w) \geq 3$ yields that $a_0(w)=a_1(w)=0$ and
that $a_2(w) \leq 0$ (Lemma \ref{lem:phi approx beta=2}). Thus
clearly $\Phi_w(n) \le \sum_{i=3}^\infty a_i(w)\frac{1}{n^i}$, and
the claim is an immediate consequence of the analysis in Lemma
\ref{lem:analyze-o(1)}. (This is true since the proof of Lemma
\ref{lem:analyze-o(1)} does not
take full advantage of the assumption that $\phi(w) = i$, but
rather that $a_j(w) \le 0$ for $j < i$.)
\end{proof}

To proceed with our analysis of Equation \eqref{eq:split-by-beta},
we now bound the number of words in ${\cal CP}_t(G)$ with
$\beta(w)=i$ for $i=0,1,2$.

\subsection{The Number of Words $w \in {\cal CP}_t(G)$ with Fixed
$\beta(w)$} \label{sbs:n_words_with_beta=}

Our proof for this bound extends an idea that originated with
\cite{buc} and was later developed in \cite{fri03}. Recall that
$\rho$ denotes the spectral radius of $T$, the universal cover of
the base graph
$G$ (as well as of any lift of $G$). Buck found a bound expressed
in terms of $\rho$ for the number of words in ${\cal CP}_t(G)$
that reduce to 1. Friedman used a similar method to bound the
number of imprimitive words in ${\cal CP}_t(G)$. We further
develop the method in order to bound the number of words in ${\cal
CP}_t(G)$ with $\beta(w)=2$.

We present the three cases ($i=0,1,2$) one by one. The case $i=0$
is indeed the simplest, and things get more complicated as $i$
grows. (However, it does seem that a general bound can be proven
for the number of words in ${\cal CP}_t(G)$ for any fixed value of
$\beta(\cdot)$.) We first note that $A_T$, the (infinite)
adjacency matrix of $T$, is a self-adjoint bounded operator on the
Hilbert space $l_2(V(T))$. Consequently, its operator norm equals its spectral
radius, i.e. $\|A_T\|=\rho$. (The same argument shows that
$\|A_T^{~l}\|=\rho^l$ for any integer $l>0$). For every
$v_1,v_2\in V(T)$ the number of paths of length $l$ from $v_1$ to
$v_2$ is $A_T^{~l}(v_1,v_2)$, which can be bounded by
$\|A_T^{~l}\| = \rho^l$.

For every $x\in V(G)$, we arbitrarily choose some vertex
$v_1=v_1(x)$ in the fiber of $x$ in $T$. For every path $\gamma$
in $G$ that starts at $x$, we consider the lift of $\gamma$ that
starts at $v_1$. We denote the tail of the lifted path by
$v_\gamma= v_\gamma(x)$.

For $i=0$, recall that $\beta(w)=0 \Leftrightarrow w$ reduces to 1
(Lemma \ref{lem:phi=beta=0}). Thus, every $w \in {\cal CP}_t(G)$
with $\beta(w)=0$ corresponds to some path in $G$ of length $t$
which reduces to 1 (a nullhomotopic path). But
these are exactly the paths which lift to closed paths in $T$ as
well. Thus:

\begin{eqnarray} \label{eq:n-words-beta=0}
|\{w\in {\cal CP}_t(G)~:~\beta(w)=0\}| &=& \sum_{x \in V(G)}
A_T^{\;t}(v_1,v_1) \leq \nonumber \\
&\leq& \sum_{x \in V(G)}\rho^t = |V(G)|\rho^t
\end{eqnarray}

For $i=1$ we want to count the number of imprimitive words in
${\cal CP}_t(G)$. If $w$ is imprimitive, then $w=u^d$ (equality in
$\mathbf{F}_k$) for some $u \in \mathbf{F}_k$ and $d \geq 2$.
Suppose that the path of the cyclically reduced form of $u$
in $G$ starts at $x \in V(G)$. Since the path of $w$ visits $x$,
there is some cyclic shift of $w$ that starts at $x$. Thus, by
adding a factor of $|w|=t$ to our eventual bound, we can assume
$w$ begins at $x$.

We now let $\gamma$ be the path in $G$ of the cyclically reduced
form of $u$ (a loop from $x$ to itself). We divide the path of $w$
through $G$ to three parts:
\begin{enumerate}
\item a path homotopic to $\gamma$ of length $l_1$ \item a path
homotopic to $\gamma$ of length $l_2$ \item a path homotopic to
$\gamma^{d-2}$ of length $l_3$
\end{enumerate}
with $l_1 + l_2 + l_3 = t$.

These three parts lift to the following paths in $T$:
\begin{enumerate}
\item a path from $v_1$ to $v_\gamma$ of length $l_1$ \item a path
from $v_1$ to $v_\gamma$ of length $l_2$ \item a path from $v_1$
to $v_{\gamma^{d-2}}$ of length $l_3$
\end{enumerate}

Thus we can bound the total number of imprimitive words in ${\cal
CP}_t(G)$ as follows:

\begin{eqnarray*}
&&|\{w\in {\cal CP}_t(G)~:~\beta(w)=1\}| \leq \\
&&\leq t \cdot \sum_{d=2}^t \sum_{x \in V(G)}
\sum_{\substack{\gamma \textrm{ is a reduced} \\ \textrm{loop from
$x$ to $x$}}} \sum_{l_1 + l_2 + l_3 = t} A_T^{~l_1}(v_1,v_\gamma)
A_T^{~l_2}(v_1,v_\gamma) A_T^{~l_3}(v_1,v_{\gamma^{d-2}})
\end{eqnarray*}
(the initial factor of $t$ accounts for the cyclic shift of $w$).

Using the symmetry of $A_T$, the inequality $A_T^{~l_3}(v_1,
v_{\gamma^{d-2}}) \leq \rho^{l_3}$ and changing the order of
summations, we obtain:
\begin{eqnarray} \label{eq:n-words-beta=1}
&&|\{w\in {\cal CP}_t(G)~:~\beta(w)=1\}| \leq \nonumber \\ &&\leq
t \cdot \sum_{d=2}^t \sum_{x \in V(G)} \sum_{l_1 + l_2 + l_3 = t}
\rho^{l_3} \sum_{v_\gamma \in Fib(x)\setminus \{v_1\} \subset
V(T)}
A_T^{~l_1}(v_1,v_\gamma) A_T^{~l_2}(v_\gamma,v_1) \leq \nonumber \\
&&\leq t(t-1) \sum_{x \in V(G)} \sum_{l_1 + l_2 + l_3 = t}
\rho^{l_3} A_T^{~l_1+l_2}(v_1,v_1) \nonumber \\
&&\leq t(t-1) \sum_{x \in V(G)} \sum_{l_1 + l_2 + l_3 = t}
\rho^{l_3} \rho^{~l_1+l_2} \nonumber \\
&& = t(t-1)|V(G)|\binom{t}{2}\rho^t \leq |V(G)|t^4\rho^t
\end{eqnarray}

(The crux of the matter in this calculation is the second step which
eliminates the need to sum over closed paths $\gamma$). \\

Next we bound the number of words in ${\cal CP}_t(G)$ with
$\beta(w)=2$. To this end we introduce a lemma that deals with quotient graphs
of smallest characteristic among all type-B quotients.
Clearly, it is such quotients that determine $\beta(w)$.

Let $G$ be a graph, $w \in {\cal CP}_t(G)$ and
$\Gamma \in {\cal Q}_w$. Each label $s_i$ that
appears in a block that is
associated with a vertex in $\Gamma$ corresponds
to a vertex in $G$. Therefore, a vertex in $\Gamma$ may correspond
to several vertices in $G$.
We next show that for $\Gamma$ as above
each vertex in $\Gamma$ corresponds to a single
vertex from $G$.

\begin{lem} \label{lem:x_in_Gamma_or_in_G}
Let $G$ be a graph and $w \in {\cal CP}_t(G)$.
If $\Gamma \in {\cal Q}_w$ has type B and $\chi(\Gamma) = \beta(w)$,
then all labels that appear in a block from
a vertex in $\Gamma$ correspond to the same vertex in $G$.
\end{lem}
\begin{proof}
The proof proceeds by showing that
otherwise there is another type-B quotient in ${\cal Q}_w$ with
smaller characteristic. Indeed,
let $\{s_{i_1},\ldots,s_{i_r}\}$ be a block in $\Gamma$
where the labels $\{s_{i_1},\ldots,s_{i_k}\}$ ($k<r$) correspond
to the vertex $v$ in $G$, while the labels $\{s_{i_{k+1}},\ldots,s_{i_r}\}$
correspond to other vertices in $G$. Define the partition
$\bar{\Gamma}$ by splitting this block to
$\{s_{i_1},\ldots,s_{i_k}\}$ and $\{s_{i_{k+1}},\ldots,s_{i_r}\}$.
We claim that this is (i) a realizable quotient in ${\cal Q}_w$
(ii) of characteristic $\chi(\Gamma)-1$ and (iii) of type B as
well.

To see (i), recall that every letter in $w$ corresponds to some
edge in $G$. Since the two parts of the split block correspond to
disjoint sets of vertices in $G$, there is no $j$ such that both
of them are heads (or tails) of a $j$-edge. Thus, the
realizability of $\Gamma$ yields the realizability of
$\bar{\Gamma}$. In deriving $\bar{\Gamma}$ from $\Gamma$
we have increased the number of vertices by one with no
additional edges,
whence (ii) is proved. To show (iii), note that any generating set of
$\bar{\Gamma}$ can be extended to a generating set of $\Gamma$ by
adding \{$s_{i_1},s_{i_{k+1}}\}$, so that if $\bar{\Gamma}$ has
type A, so does $\Gamma$.
\end{proof}


As we saw (Lemmas
\ref{lem:beta-invrt-cyc-shft} and \ref{lem:beta-invrt-red})
if $w'$ is the cyclic reduction of $w$, then $\beta(w') = \beta(w)$.
Moreover, every type-B $\Gamma \in
{\cal Q}_w$ with $\chi(\Gamma)=\beta(w)$ can be generated from
some $\Gamma' \in {\cal Q}_{w'}$ with $\chi(\Gamma') =
\chi(\Gamma)$, through a series of $\delta$-operations (as
in Lemma \ref{lem:beta-invrt-red}). Since $w'$ is cyclically
reduced, every vertex in $\Gamma'$ has degree $\geq 2$. (Indeed,
$\Gamma'$ is obtained from $\Gamma$ by successive elimination of
vertices of degree one in the graph).

Now assume $\beta(w)=\chi(\Gamma)=2$. There are three possible
shapes that $\Gamma'$ can have: Figure-Eight, Barbell or Theta
(see Figure \ref{fig:shapes-of-chi=2}). For a cost of an
additional factor of $t$ as above, we may assume the path of $w$
through $\Gamma$ begins at the vertex $x$ specified in each of the
diagrams. (More accurately, it begins at the vertex of $\Gamma$
corresponding to $x$ through the series of $\delta$-operations.)
By Lemma \ref{lem:x_in_Gamma_or_in_G}, $x$ corresponds to a unique
vertex in $G$ which we call $x$ as well (by abuse of notation).
This vertex $x$ in $G$ marks the starting point of $w$. We now
analyze each case separately, and using the notations in Figure
\ref{fig:shapes-of-chi=2}, we trace the path of $w'$ through
$\Gamma'$.

\begin{figure}[htb]
\centering
\includegraphics[width=1\textwidth]{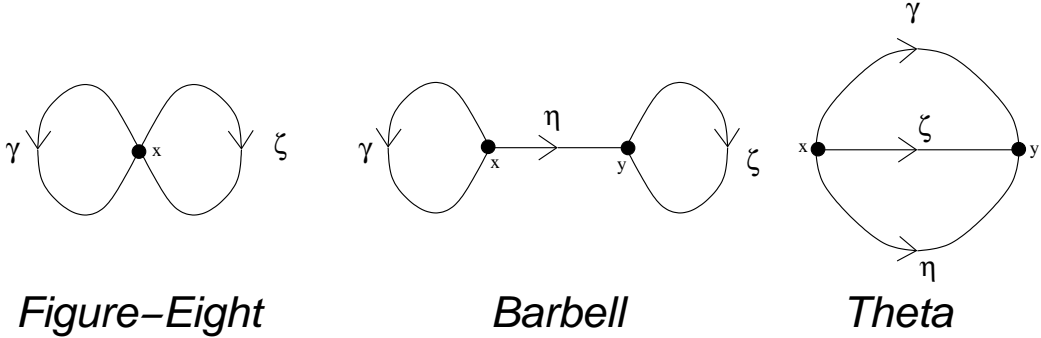}
\caption{The three possible shapes of quotient graphs of a
cyclically reduced word with characteristic 2. The edges $\gamma$,
$\zeta$ and $\eta$ denote subwords.} \label{fig:shapes-of-chi=2}
\end{figure}

Assume first that $\Gamma'$ has the shape
of a Figure-Eight. In this case $w'$ can be expressed using
$\gamma,\zeta$ in a reduced expression in which each of them
appears at least twice (Lemma \ref{lem:2-preimages}). For any fixed
reduced expression in $\gamma,\zeta$, we specify certain two
appearances of $\gamma$ and certain two appearances of $\zeta$.
The path of $w$ through $\Gamma$ can be then divided to at most
seven parts: four parts for the chosen appearances of $\gamma$ and
$\zeta$, and three parts for sequences of the rest of the
expression (we can always choose the
first two characters in the expression, but we may
be forced to have spaces
between the second and the third, between the third and the fourth
and after the fourth). We then proceed to a calculation
as in \eqref{eq:n-words-beta=1}.

To illustrate, let $w'=\gamma \gamma \gamma \zeta \gamma^{-1}
\zeta^{-1} \zeta^{-1} \gamma$. We split $w'$ to seven parts as
shown in the following bracketing: $w'=(\gamma)(\gamma)(\gamma)
(\zeta)(\gamma^{-1})(\zeta^{-1})(\zeta^{-1} \gamma)$. The
corresponding lengths are: $l_1$ steps for the first $\gamma$,
$l_2$ for the second, $l_3$ steps for the next $\gamma$ (which is
considered ``a space''), $l_4$ steps for $\zeta$, $l_5$ for the
space $\gamma^{-1}$, $l_6$ for $\zeta^{-1}$ and $l_7$ for the
space $\zeta^{-1} \gamma$. We can now bound the total number of
words which reduce cyclically (and with a possible cyclic shift)
to this expression in some $\gamma$ and $\zeta$. (The first factor
of $t$ in the calculation
accounts for the initial cyclic shift of $w$.)

\begin{eqnarray*}
& |\{w \in {\cal CP}_t(G)~:~w\textrm{ reduces cyclically to }
\gamma \gamma \gamma \zeta \gamma^{-1} \zeta^{-1}
\zeta^{-1} \gamma \textrm{ for some } \gamma, \zeta \}|& \\
 &\leq t \cdot \sum_{x \in V(G)} \sum_{\gamma,\zeta} \sum_{l_1 + \ldots + l_7 = t}
 A_T^{~l_1}(v_1,v_\gamma)
A_T^{~l_2}(v_1,v_\gamma) A_T^{~l_3}(v_1,v_\gamma)
A_T^{~l_4}(v_1,v_\zeta)  \cdot &\\
& \cdot A_T^{~l_5}(v_\gamma,v_1) A_T^{~l_6}(v_\zeta,v_1)
A_T^{~l_7}(v_1,v_{\zeta^{-1} \gamma})&
\end{eqnarray*}
(The second sum is over all $\gamma$ and $\zeta$ -
two distinct non-empty reduced loops from $x$ to itself in $G$.)
We proceed as before:

\begin{eqnarray*}
\leq t \cdot \sum_{x \in V(G)} \sum_{l_1 + \ldots + l_7 = t}
\rho^{l_3 + l_5 + l_7} & \sum_{v_\gamma \in Fib(x)\setminus
\{v_1\}} A_T^{~l_1}(v_1,v_\gamma) A_T^{~l_2}(v_\gamma,v_1) \cdot
\\ \cdot & \sum_{v_\zeta \in Fib(x)\setminus \{v_1\}} A_T^{~l_4}(v_1,v_\zeta)
A_T^{~l_6}(v_\zeta,v_1) \leq \\
\leq t \cdot \sum_{x \in V(G)} \sum_{l_1 + \ldots + l_7 = t}
\rho^{l_3 + l_5 + l_7} &\rho^{l_1 + l_2} \rho^{l_4 + l_6} \leq
|V(G)|t^7 \rho^t
\end{eqnarray*}

This bound was calculated for a specific reduced expression in
$\gamma, \zeta$. The total number of possible expressions is less
than $3^t$. (w.l.o.g every expression begins with $\gamma$, and it
contains a total of between four and $t$ components. For $r$
components there are at most $3^{r-1}$ possible continuations, and
$3^3 + 3^4 + \ldots + 3^{t-1} < 3^t$). Thus we can bound the total
number of words in ${\cal CP}_t(G)$ with $\beta(w)=2$ and which
have a type-B Eight-Figure quotient graph, by $|V(G)|t^7 3^t
\rho^t$.

For the Barbell and Theta the analysis is similar, but their
contribution is negligible relative to the contribution of the
Figure-Eight. This time we construct a reduced expression in three
subwords: $\gamma, \zeta$, and $\eta$, but the possible number of
expressions is bounded by $2^t$ (the same argument as above, only
this time every subword has only two possible subsequent
subwords). We need to specify two occurrences of each of the three
letters this time, so we may need to split the path of $w'$ to
$6+5=11$ parts. The bound is therefore $|V(G)|t^{11} 2^t \rho^t$,
the asymptotic comparison is clearly $|V(G)|t^{11} 2^t \rho^t \ll
|V(G)|t^7 3^t \rho^t$.

To illustrate the calculation, consider the following (reduced) expression
for the Barbell: $w' = \gamma \eta \zeta \zeta \zeta \eta^{-1}
\gamma^{-1} \eta \zeta \eta^{-1} \gamma$. The number of words in
${\cal CP}_t(G)$ which reduce cyclically to such an expression can
be bounded by:
\begin{eqnarray*}
\leq& t \cdot \sum_{x,y,\gamma,\zeta,\eta} \sum_{l_1, \ldots,l_8}&
A_T^{~l_1}(v_1(x),v_\gamma(x)) A_T^{~l_2}(v_1(x),v_\eta(x))
A_T^{~l_3}(v_1(y),v_\zeta(y)) \cdot \\
& &\cdot A_T^{~l_4}(v_1(y),v_\zeta(y))
A_T^{~l_5}(v_1(y),v_\zeta(y)) A_T^{~l_6}(v_\eta(x),v_1(x)) \cdot \\
&& \cdot A_T^{~l_7}(v_\gamma(x),v_1(x))
A_T^{~l_8}(v_1(x),v_{\eta \zeta \eta^{-1} \gamma}(x)) \\
\leq& t \cdot \sum_{l_1, \ldots,l_8} \rho^{l_5 + l_8} & \sum_x
\rho^{l_1 + l_7} \sum_y \rho^{l_3+ l_4} \sum_{v_\eta \in Fib(y)}
A_T^{~l_2}(v_1(x),v_\eta) A_T^{~l_6}(v_\eta,v_1(x)) \\
\leq& t \cdot \sum_{l_1, \ldots,l_8} \rho^{t - l_2 - l_6} & \sum_x
\sum_{v_\eta \in V(T)}
A_T^{~l_2}(v_1(x),v_\eta) A_T^{~l_6}(v_\eta,v_1(x)) \\
\leq& |V(G)| t^8 \rho^t
\end{eqnarray*}
(Here $x$ and $y$ are vertices of $G$, $\gamma$ (resp. $\zeta$)
is a reduced loop from $x$ (resp. $y$) to itself, and
$\eta$ a reduced path from $x$ to $y$, and $l_1 + \ldots +l_8 =
t$).

To sum up, we have for large enough $t$:
\begin{eqnarray} \label{eq:n-words-beta=2}
|\{w\in {\cal CP}_t(G)~:~\beta(w)=2\}| \leq (1+o(1)) |V(G)| t^7
3^t \rho^t
\end{eqnarray}

\begin{remark} \label{rem:bound_for_beta=r}
These calculations suggest that for every integer $r\geq2$, the
dominant figure among quotient graphs of characteristic $r$ (after
cyclic reduction) is a bouquet with $r$ loops. Thus, for large
enough $t$, the number of words in ${\cal CP}_t(G)$ with
$\beta(w)=r$ is less than $(1+o(1)) |V(G)| t^{4r-1} (2r-1)^t
\rho^t$.
\end{remark}

\begin{remark}
Note that this counting argument is quite wasteful, and involves a
good deal of overcounting. For instance, in the case $i=1$, we
counted each word of the form $w=u^4$ twice: once with the root
$u$ and $d=4$, and once with the root $u^2$ and $d=2$. It seems
that in
fact, we have bounded $\sum_{\substack{w \in {\cal CP}_t(G)\\
\beta(w)=i}} a_i(w)$.
\end{remark}

\subsection{The Proof of Theorem \ref{ther:2/3}}
We now have all the necessary
ingredients for a proof of Theorem
\ref{ther:2/3}. Recall that in \eqref{eq:split-by-beta}, we split
the set of words ${\cal CP}_t(G)$ to four subsets according to the
value of $\beta(\cdot)$. In Table \ref{tab:split-by-beta}, we
collect the following information for each subset: A bound 
on the subset's size
and a bound on the value of $\Phi_w(n)$ for the words in the
subset. This table highlights
the significance of the proved and unproved relations
between $\phi(\cdot)$ and $\beta(\cdot)$ to the analysis of the sum in
\eqref{eq:split-by-beta}. The value of $\phi(\cdot)$ yields
the bounds in the right column of the table, whereas
$\beta(\cdot)$ is used to derive the bounds in the middle one.

The number of words with $\beta(w)=0$ was bounded in
\eqref{eq:n-words-beta=0}, and the bound for $\Phi_w(n)$ is from
Lemma \ref{lem:analyze-o(1)} (since $a_0(w)=0$
whenever $\beta(w)=0$). In \eqref{eq:n-words-beta=1} we
bounded the number of imprimitive words in ${\cal CP}_t(G)$, and
Lemma \ref{lem:analyze-o(1)} yields again a bound for $\Phi_w(n)$
for imprimitive $w$. (By Lemma \ref{lem:n_manot}, $a_1(w) \leq t^2$).
The number of words with $\beta(w)=2$ was
bounded in \eqref{eq:n-words-beta=2} (for $t$ large enough), and
this time we bound $\Phi_w(n)$ using the fact that $\phi(w)\geq 2$ for
every word in this subset. The size of the fourth set is bounded
by the total size of ${\cal CP}_t(G)$
\[|{\cal CP}_t(G)| = tr(A_G^{\;~t}) =
\sum_{i=1}^{|V(G)|}\lambda_i^{\;t} \leq |V(G)|\lambda_1^{\;t},\]
and the bound on $\Phi_w(n)$ in this case comes from Lemma
\ref{lem:beta>=3} and the analysis of $a_i(w)$ in the proof of
Lemma \ref{lem:analyze-o(1)}.

\begin{table}[htb]
\centering
\begin{tabular}[c]{c|c|c}
 The Value of $\beta(w)$ & The size of the subset $\leq$ & $\Phi_w(n) \leq $ \\
\hline 0  & $|V(G)|\cdot \rho^t$ & $1 + \frac{t^4}{n}$\\
1 & $|V(G)|t^4 \cdot \rho^t$ & $\frac{1}{n}\left( t^2 + \frac{t^6}{n} \right)$ \\
2 & $|V(G)|t^7 \cdot 3^t \cdot \rho^t$ & $\frac{1}{n^2}\left( t^4 + \frac{t^8}{n} \right) $ \\
$\geq 3$ & $|V(G)|\cdot \lambda_1^{\;t}$ & $\frac{1}{n^3}\left( 3t^6 + \frac{t^{10}}{n} \right) $ \\

\end{tabular} \caption{A bound for the size of each of the four
subsets of ${\cal CP}_t(G)$, and a bound for the value of
$\Phi_w(n)$ for each word in the subset.}
\label{tab:split-by-beta}
\end{table}

We now select $t$ to minimize the resulting upper bound on
$\mathbb{E}(\mu_{max}^{\;~~~t})$. It is not hard to see that the optimal
$t$ is $\Theta(\log n)$. Consequently, all terms $\frac{t^4}{n},
\ldots, \frac{t^{10}}{n}$ are $o_n(1)$ and can be ignored. The
information in Table \ref{tab:split-by-beta} is now combined with
\eqref{eq:split-by-beta} to yield [for $t=\Theta(\log n)]$:

\begin{eqnarray*}
\mathbb{E}(\mu_{max}^{\;~~~t}) & \leq & (1+o_n(1))\cdot \\
&&\left( n \cdot |V(G)|\cdot
\rho^t +  |V(G)|\cdot t^6 \cdot \rho^t + |V(G)|\cdot t^{11} \cdot 3^t \cdot \frac{\rho^t}{n} +
|V(G)| \cdot 3t^6 \cdot \frac{\lambda_1^{\;t}}{n^2} \right) \\
& \leq & (1+o_n(1))|V(G)|\cdot t^{11}\cdot4\left[max(\rho\cdot n^{1/t},
\rho,3\rho\cdot n^{-1/t}, \lambda_1\cdot n^{-2/t})\right]^t
\end{eqnarray*}

In order to balance the first and the last terms, we set $n^{1/t} \approx
\rho^{-1/3}\lambda^{1/3}$ (recall that $t$ is an even integer, so
we cannot guarantee exact equality here). Since $t \to \infty$ with
$n$, the constant and polynomial factors can be replaced
by $(1+o_n(1))^t$. We obtain

\begin{eqnarray*}
\mathbb{E}(\mu_{max}^{\;~~~t}) &\leq&
\left[\lambda_1^{\;1/3} \rho^{2/3} \cdot \left(1+o_n(1)\right) \cdot
 max\left(1, \left(\frac{\rho}{\lambda_1}\right)^{1/3},
3\left(\frac{\rho}{\lambda_1}\right)^{2/3}, 1\right)\right]^t \\
&=& \left[\lambda_1^{\;1/3} \rho^{2/3} \cdot \left(1+o_n(1)\right) \cdot
 max\left(1, 3\left(\frac{\rho}{\lambda_1}\right)^{2/3}\right)\right]^t
\end{eqnarray*}
(the equality holds because $\rho \leq \lambda_1$ is always true, see
Remark \ref{rem:rho<=lambda1}).

The statement of Theorem \ref{ther:2/3} now follows from a standard application
of Markov's inequality. Obviously, for every $\epsilon>0$,
$3\cdot\lambda_1^{\;1/3} \rho^{2/3}+\epsilon$
may serve as an absolute upper bound.

\begin{remark}
Here is a sketch of our proposed approach to Friedman's
Conjecture. Say we seek to prove that all new eigenvalues are,
almost surely $O(\rho^{1-\epsilon} \lambda_1^\epsilon)$ for every
$\epsilon>0$. To prove a bound of $O(\rho^\frac{r}{r+1}
\lambda_1^\frac{1}{r+1})$ one would have to show that $\beta(w)=i
\Leftrightarrow \phi(w)=i$ for every $i \leq r$. We believe that
this can be shown in a way similar to our proof of Theorem
\ref{ther:2/3}. In addition, it would be necessary to follow up on
Remark \ref{rem:bound_for_beta=r}, and establish a bound on the
number of words $w \in {\cal CP}_t(G)$ with given $\beta(w)$.
\end{remark}


\section{The Number of $L$-Cycles in $w(s_1,\ldots,s_k)$} \label{sec:nica}

In this section we introduce a new conceptual and relatively
simple proof of a Theorem of A. Nica \cite{nica}. In
\eqref{eq:xwn} we defined the random variable $X_w^{(n)}$ which
counts the number of fixed points in $w(\sigma_1,\ldots,\sigma_k)$
for fixed $w$. We extend this concept and for every integer
$L\geq1$ denote by $X_{w,L}^{(n)}$ a random variable on
$S_n^{\;k}$ which is defined by:
\begin{equation} \label{eq:xwln}
X_{w,L}^{(n)}(\sigma_{1},\ldots,\sigma_{k}) = \textrm{\# of cycles
of length $L$ of } w(\sigma_{1}, \ldots , \sigma_{k})
\end{equation}
($X_{w,1}^{(n)}$ is a new notation for $X_w^{(n)}$). \\

Nica's theorem says that the variables $X_{w,L}^{(n)}$ have, for
fixed $w$ and $L$ and for $n\to\infty$, a limit distribution which
can be computed explicitly. (Unless otherwise stated, the
distribution on $S_{n}^{\;k}$ is always uniform.)

\begin{ther} \label{ther:limit-xwln}
Let $1 \neq w \in \mathbf{F}_{k}$ and suppose that $w=u^{d}$, with
$u$ primitive. Then for every integer $L \geq 1$, the random
variable $X_{w,L}^{(n)}$ defined in (\ref{eq:xwln}) has, for
$n\to\infty$, a limit distribution, which is given by:
\begin{equation} \label{eq:conv}
X_{w,L}^{(n)} \stackrel{dis}{\to} \sum_{h\in H(d,L)}hZ_{1/Lh}
\end{equation}
where $H(d,L)$ is the set
\begin{equation} \label {eq:hdl}
H(d,L)=\{h>0: ~~ h|~d \textrm{ and gcd}(\frac{d}{h},L)=1\},
\end{equation}
$Z_{m} \sim Poi(m)$ (a variable with Poisson distribution with
parameter $m$), and ``$\stackrel{dis}{\to}$'' denotes convergence
in distribution.

In particular, this limit distribution depends only on $d$ and $L$
(and not on $u$).
\end{ther}

Note that in the case that $w$ is primitive (i.e. $d=1$), the
limit distribution is simply Poisson with parameter $1/L$.

Our proof of Theorem \ref{ther:limit-xwln} is based on the method
of moments and provides, in particular, explicit expressions for
the moments of $X_{w,L}^{(n)}$:

\begin{cor} \label{cor:moments}
For every $1\neq w \in \mathbf{F}_{k}$, $L \geq 1$ and $r \geq 1$
there exists a rational function $\Psi_{w,L,r}$ that is defined on
a neighborhood of $0$ such that
$\mathbb{E}([X_{w,L}^{(n)}]^r)=\Psi_{w,L,r}(1/n)$ for sufficiently
large $n$. Consequently, the limit $\lim_{n \to
\infty}\mathbb{E}([X_{w,L}^{(n)}]^r)$ exists and equals
$\Psi_{w,L,r}(0)$. In addition, if $w=u^{d}$, with $u$ primitive,
then this limit equals the corresponding moment of the sum in
\eqref{eq:conv}. In particular, $\Psi_{w,L,1}(0) = \lim_{n \to
\infty}\mathbb{E}(X_{w,L}^{(n)})=\frac{|H(d,L)|}{L}$.
\end{cor}
(Note that this is a generalization of the function $\Phi_w(n)$
defined in \eqref{eq:Phi}: $n\cdot\Phi_w(n)+1 =
\mathbb{E}(X_{w,1}^{(n)})=\Psi_{w,1,1}(1/n)$.)

The contents of Nica's work is Theorem \ref{ther:limit-xwln} and
Corollary \ref{cor:moments} for the first moment. \\

The main idea of the proof was already illustrated in Section
\ref{sec:word-maps}, where we analyzed the expectation of
$X_{w,1}^{(n)}$. Recall Equation~\eqref{eq:Phi} where $\Phi_w(n)$
is expressed as a power
series $\sum_{i=0}^\infty a_i(w)\frac{1}{n^i}$.
We already know (Lemma \ref{lem:phi=beta=0}) that when $w\neq 1$,
$a_0(w)=0$, and so $\lim_{n\to \infty} \mathbb{E}(X_{w,1}^{(n)}) =
a_1(w) + 1$. But this is exactly the
number of graphs $\Gamma \in {\cal Q}_w$ with $\chi(\Gamma)=1$
(see, for instance, Lemma \ref{lem:phi=beta=1}).

If $w$ is cyclically reduced, then the only graphs in ${\cal Q}_w$
with $\chi=1$ are cycles (again, see the proof of Lemma \ref{lem:phi=beta=1}). Such a cycle consists of a cyclic
concatenation of several copies of $u$ (the primitive word such
that $w=u^d$). The number of copies of $u$ in the cycle has to
divide $d$, hence
\[ \lim_{n\to\infty}\mathbb{E}(X_{w,1}^{(n)}) = \big| H(d,1) \big|\ =
\textrm{ \# of divisors of } d\]

But we can indeed restrict our discussion to cyclically reduced
words. The justification for this is the following. Let $1 \neq
w\in \mathbf{F}_k$, and let $w'$ be its cyclic reduction. We have
already mentioned that $w \sim w'$ and so they induce the same
distribution on $S_n$. In particular, $X_{w,L}^{(n)}$ and
$X_{w',L}^{(n)}$ are equally distributed. It is also quite evident
that $w$ and $w'$ share an identical exponent of their primitive
root (i.e., if $w=xw'x^{-1}$ for some $x\in \mathbf{F}_k$ and
$w'=u^d$ with $u$ primitive, then $xux^{-1}$ is primitive too, and
$w=(xux^{-1})^d$). Thus the validity of Theorem
\ref{ther:limit-xwln} and Corollary \ref{cor:moments} for $w'$,
yields their validity for $w$ as well.

Below, we extend this argument to obtain the limit of the
expectation of $\big[X_{w,L}^{(n)}\big]^r$ for every $L$ and $r$.
We essentially use the same way we counted fixed points in
$w(\sigma_1,\ldots,\sigma_k)$ for all $k$-tuples
$(\sigma_1,\ldots,\sigma_k) \in S_n^{\;k}$, to count $L$-cycles
and \emph{sequences} of $L$-cycles, which we call \emph{lists} of
cycles. The point is that the total number of lists of length $r$
of $L$-cycles, divided by $(n!)^k$, equals the expectation of
$\big[X_{w,L}^{(n)}\big]_r$, the $r$-th \emph{factorial} moment of
$X_{w,L}^{(n)}$. Once we know how to calculate the limits of the
factorial moments, the limits of the regular moments are at easy
reach. To finish, we show that these limits equal the
corresponding moments in the r.h.s of \eqref{eq:conv}, and use the
method of moments to conclude the proof.

\subsection{Lists of Trails and their Categories}
\label{sbs:categories}

We begin by generalizing some of the notions from Section
\ref{sbs:exp-fix-points}. Let $1 \neq w \in \mathbf{F}_k$ be cyclically
reduced, $n \geq 1$ an integer, $s_0 \in \{1,\ldots,n\}$, and
$\sigma_1,\ldots,\sigma_k \in S_n$. The \emph{trail of $s_0$
through $w(\sigma_1,\ldots,\sigma_k)$} is the sequence of images
of $s_0$ under $w(\sigma_1,\ldots,\sigma_k)$. Namely, if
$w=g_{i_{1}}^{\;\alpha_{1}}g_{i_{2}}^{\;\alpha_{2}}\ldots
g_{i_{m}}^{\;\alpha_{m}}$ (with $\alpha_i \in \{-1,1\}$) in
reduced form, then associated with $s_0$ is the following path:
\begin{displaymath}
s_0 \stackrel{\sigma_{i_1}^{\;\alpha_1}}{\longrightarrow} s_1
\stackrel{\sigma_{i_2}^{\;\alpha_2}}{\longrightarrow} s_2
\stackrel{\sigma_{i_3}^{\;\alpha_3}}{\longrightarrow} \ldots
\stackrel{\sigma_{i_m}^{\;\alpha_m}}{\longrightarrow} s_m
\end{displaymath}
with $s_1,\ldots,s_m \in \{1,\ldots,n\}$, and $s_b =
\sigma_{i_b}^{\;\alpha_b}(s_{b-1})$ ($b=1,\ldots,m$). (Recall that
we compose permutations from left to right.)

Likewise, we can speak of the trail through some power of $w$. For
example, the trail of $s_0$ through
$w^3(\sigma_1,\ldots,\sigma_k)$ is
\begin{displaymath}
s_0 \stackrel{\sigma_{i_1}^{\;\alpha_1}}{\longrightarrow} s_1
\stackrel{\sigma_{i_2}^{\;\alpha_2}}{\longrightarrow} \ldots
\stackrel{\sigma_{i_m}^{\;\alpha_m}}{\longrightarrow} s_m
\stackrel{\sigma_{i_1}^{\;\alpha_1}}{\longrightarrow} s_{m+1}
\stackrel{\sigma_{i_2}^{\;\alpha_2}}{\longrightarrow} \ldots
\stackrel{\sigma_{i_m}^{\;\alpha_m}}{\longrightarrow} s_{2m}
\stackrel{\sigma_{i_1}^{\;\alpha_1}}{\longrightarrow} s_{2m+1}
\stackrel{\sigma_{i_2}^{\;\alpha_2}}{\longrightarrow} \ldots
\stackrel{\sigma_{i_m}^{\;\alpha_m}}{\longrightarrow} s_{3m}
\end{displaymath}
with $s_1,\ldots,s_m$ as before, and $s_{m+1},\ldots,s_{3m} \in
\{1,\ldots,n\}$ satisfying the obvious constraints.

We recall that two trails were placed in the same category
if they have the same coincidence pattern.
This notion can be extended to our present, more
general context, in an obvious way. Namely,
every category is associated with some directed edge-colored
graph. Moreover, we can define categories not only of single
trails, but also of \emph{lists} of trails, and again, associate a
graph to each category. The nature of this graph is exactly as
described in Section \ref{sbs:exp-fix-points}.

To illustrate, let $w=g_1g_2g_1^{\;-1}g_2^{\;-1}$ be the
commutator word, $n \geq 8$ an integer and $\sigma_1,\sigma_2 \in
S_n$ such that the following trails are realized by
$w(\sigma_1,\sigma_2)$ and $w^2(\sigma_1,\sigma_2)$, respectively:
\begin{displaymath}
\begin{array}{c}
1 \stackrel{\sigma_1}{\longrightarrow} 3
\stackrel{\sigma_2}{\longrightarrow} 7
\stackrel{\sigma_1^{\;-1}}{\longrightarrow} 3
\stackrel{\sigma_2^{\;-1}}{\longrightarrow} 8 \\
\qquad 4 \stackrel{\sigma_1}{\longrightarrow} 6
\stackrel{\sigma_2}{\longrightarrow} 8
\stackrel{\sigma_1^{\;-1}}{\longrightarrow} 5
\stackrel{\sigma_2^{\;-1}}{\longrightarrow} 5
\stackrel{\sigma_1}{\longrightarrow} 8
\stackrel{\sigma_2}{\longrightarrow} 3
\stackrel{\sigma_1^{\;-1}}{\longrightarrow} 1
\stackrel{\sigma_2^{\;-1}}{\longrightarrow} 2
\end{array}
\end{displaymath}
We denote the nodes of the trail through $w$ by
$s^1_0,\ldots,s^1_4$ and the nodes through $w^2$ by
$s^2_0,\ldots,s^2_8$. Then the graph associated with
the category of this list of trails through $w,w^2$
is shown in Figure \ref{fig:first-graph}. \\

\begin{figure}[htb]
\centering
\includegraphics[width=0.8\textwidth]{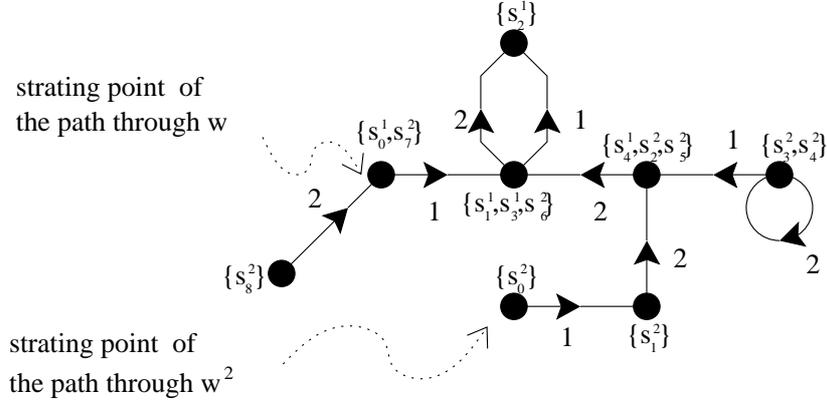}
\caption{An example of a graph associated with a certain category of
a list of trails through $w,w^2$, where
$w=g_1g_2g_1^{\;-1}g_2^{\;-1}$ is the commutator word. }
\label{fig:first-graph}
\end{figure}

Although the notions here have a wider scope, we limit our
discussion to categories of trails which represent \emph{cycles}
in $w(\sigma_1,\ldots,\sigma_k)$: an $L$-cycle is represented by a
closed trail through $w^L(\sigma_1,\ldots,\sigma_k)$. In fact, we
are interested in counting cycles of a given length, so for our
purposes we can confine ourselves to categories of a list of $r$
trails through $w^L(\sigma_1,\ldots,\sigma_k)$, for some $L,r \geq
1$.

First, let us analyze the categories that represent a single
$L$-cycle. A closed trail through $w^L(\sigma_1,\ldots,\sigma_k)$
represents an $L$-cycle only if it does not represent any smaller
cycle. That is, in the closed trail
\begin{displaymath}
s_0 \stackrel{\sigma_{i_1}^{\;\alpha_1}}{\longrightarrow} s_1
\stackrel{\sigma_{i_2}^{\;\alpha_2}}{\longrightarrow} \ldots
\stackrel{\sigma_{i_m}^{\;\alpha_m}}{\longrightarrow} s_m
\stackrel{\sigma_{i_1}^{\;\alpha_1}}{\longrightarrow} \ldots
\stackrel{\sigma_{i_m}^{\;\alpha_m}}{\longrightarrow} s_{2m}
\stackrel{\sigma_{i_1}^{\;\alpha_1}}{\longrightarrow} \ldots
\ldots
\stackrel{\sigma_{i_{m-1}}^{\;\alpha_{m-1}}}{\longrightarrow}
s_{Lm-1} \stackrel{\sigma_{i_m}^{\;\alpha_m}}{\longrightarrow}
s_0,
\end{displaymath}
the $L$ labels $s_0,s_m,s_{2m},\ldots,s_{(L-1)m}$ must all be
distinct.

Once again, the graph of each category is a quotient of the
\emph{universal graph}. In this case, the universal graph is
$C_{L\cdot m}$, the cycle of length $L \cdot m$.
A partition of its vertices corresponds to some category of an
$L$-cycle if and only if it does not create ``collisions'' of
same-color edges, and keeps apart all $L$ vertices corresponding
to $s_0, s_m, s_{2m}, \ldots, s_{(L-1)m}$. We demonstrate this in
Figure \ref{fig:L-univ-graph}.
\\

\begin{figure}[htb]
\centering
\includegraphics[width=0.6\textwidth]{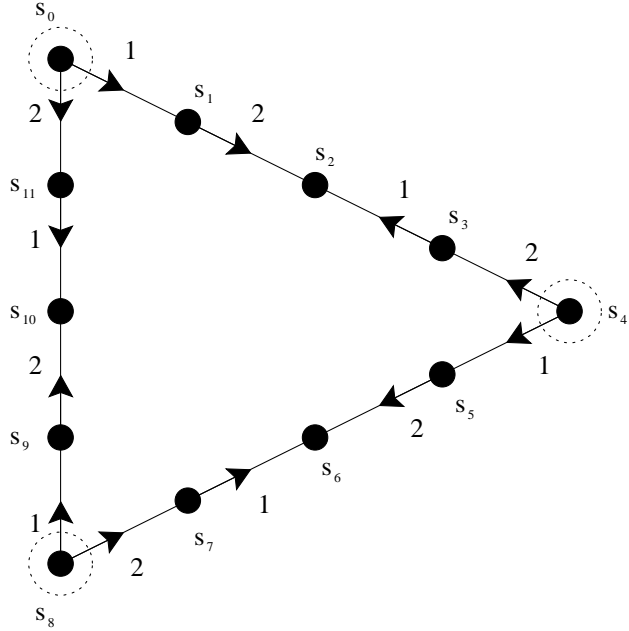}
\caption{The universal graph of a 3-cycle when
$w=g_1g_2g_1^{\;-1}g_2^{\;-1}$. In a realizable partition of the
vertices that represents a 3-cycle, the 3 circled vertices ($s_0,s_4,s_8$) must
belong to distinct blocks.} \label{fig:L-univ-graph}
\end{figure}

The most general case in our discussion comes up when we turn to
calculate the $r$-th moment of $X_{w,L}^{(n)}$. To this end we
consider categories of lists of $r$ $L$-cycles through
$w(\sigma_1, \ldots, \sigma_k)$. The universal graph
$\widetilde{\Gamma}_{w,L,r}$ which represents $r$ ordered cycles
of length $L$ each, consists of a disjoint union of $r$ copies of
$C_{L\cdot m}$. We name the vertices of the first cycle in $\widetilde{\Gamma}_{w,L,r}$ by
$s^1_0,\ldots,s^1_{Lm-1}$, the vertices of the second cycle by $s^2_0,\ldots,s^2_{Lm-1}$
and so on until $s^r_0,\ldots,s^r_{Lm-1}$ for the $r$-th cycle.

We are interested in quotients (or partitions of the vertices) of
$\widetilde{\Gamma}_{w,L,r}$ that represent $r$ distinct
$L$-cycles. This means that the vertices
$$s^1_0, s^1_m, s^1_{2m}, \ldots, s^1_{(L-1)m}, s^2_0, s^2_m\ldots, s^2_{(L-1)m}, \ldots, s^r_0, s^r_m, \ldots, s^r_{(L-1)m}$$
(a total of $L \cdot r$ vertices) should be
kept apart in every partition, as illustrated in Figure
\ref{fig:Lr-univ-graph}.

\begin{figure}[htb]
\centering
\includegraphics[width=0.8\textwidth]{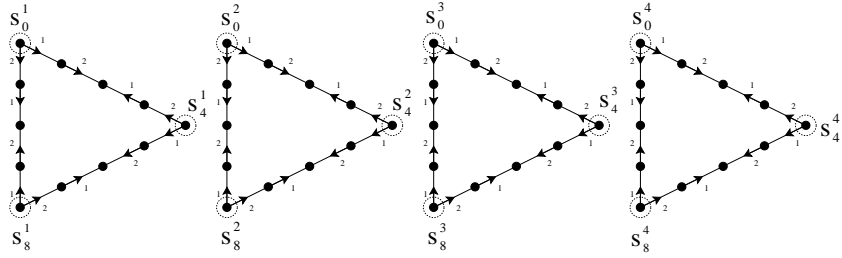}
\caption{The universal graph $\widetilde{\Gamma}_{w,3,4}$
of four 3-cycles where $w=g_1g_2g_1^{\;-1}g_2^{\;-1}$. In
a realizable partition of the vertices that represents four
3-cycles, the 12 circled vertices must belong to distinct blocks.
(We omit the blocks corresponding to the rest of the vertices.)}
\label{fig:Lr-univ-graph}
\end{figure}

\subsection{Formulas for the Moments of $X_{w,L}^{(n)}$}
\label{sbs:formulas}

The method we use to calculate the expectation of
$[X_{w,L}^{(n)}]^r$ is quite
similar to the one we used for $\mathbb{E} \big(
X_{w,1}^{(n)} \big)$. We first calculate the ``factorial moments''
of $X_{w,L}^{(n)}$ and derive the regular moments from them (the
$r$-th factorial moment of a random variable $X$ is defined as
$\mathbb{E} \big( [X]_r \big)$, where $[X]_r$ is the ``falling
factorial'', namely $[X]_r = X(X-1) \ldots (X-r+1)$~).

$\left[X_{w,L}^{(n)}\right]_r$ counts lists of $r$ $L$-cycles in
$w(\sigma_1, \ldots, \sigma_k)$. As in the case of the first
moment, we calculate its expectation by counting the total number
of lists of $r$ $L$-cycles in $w(\sigma_1, \ldots, \sigma_k)$ for
\emph{all} $k$-tuples $(\sigma_1, \ldots, \sigma_k) \in S_n^{\;k}$
and dividing by $(n!)^k$.

The counting is carried out by classifying these lists into
categories. Each list of $r$ $L$-cycles \emph{with specified
starting point for each cycle}, belongs to some category. These
categories are the quotients of the universal graph
$\widetilde{\Gamma}_{w,L,r}$, which we denote by ${\cal
Q}_{w,L,r}$ (e.g., ${\cal Q}_{w,1,1}$ is the same set as ${\cal
Q}_w$). To recap, the set ${\cal Q}_{w,L,r}$ can be generated as follows:

We first draw $\widetilde{\Gamma}_{w,L,r}$, the
universal graph of $r$ ordered $L$-cycles of $w$, which consists
of $r$ disjoint cycles each of which has $L \cdot |w|$ vertices.
${\cal Q}_{w,L,r}$ consists of quotient graphs
that are generated by partitions of the vertices of
$\widetilde{\Gamma}_{w,L,r}$. A quotient graph is included in
${\cal Q}_{w,L,r}$ if it is realizable, and if in the
corresponding partition each of the $r \cdot L$ vertices that
represent the $r$ $L$-cycles is in a different block. \\

Let $\Gamma$ be some graph in ${\cal Q}_{w,L,r}$. A
\emph{realization} of $\Gamma$ is a $k$-tuple of permutations
$\sigma_1,\ldots,\sigma_k \in S_n$, a list of $r$ $L$-cycles of
$w(\sigma_1,\ldots,\sigma_k)$ and a specified starting point for
each cycle, such that they belong to $\Gamma$'s category. The
number of realizations of $\Gamma$ is the same as in
\eqref{eq:n-realizations}, namely:
\begin{equation} \label{eq:n-realizations-2}
N_\Gamma(n) =
n(n-1)\ldots(n-v_\Gamma+1)\prod_{j=1}^{k}\big(n-e^j_\Gamma\big)!
\end{equation}

Since every list of $r$ $L$-cycles is counted $L^r$ times (there
are $L^r$ ways to choose the starting points), we have:

\begin{eqnarray} \label{eq:moment-formula1}
\mathbb{E}\big([X_{w,L}^{(n)}]_r\big) &=& \frac{1}{(n!)^k}
\sum_{(\sigma_1,\ldots,\sigma_k)\in S_n^{\;k}}
[X_{w,L}^{(n)}(\sigma_1,\ldots,\sigma_k)\big]_r = \nonumber \\
&=& \frac{1}{(n!)^k} \cdot \frac{1}{L^r} \sum_{\Gamma \in {\cal
Q}_{w,L,r}}N_\Gamma(n) = \\  \label{eq:moment-formula2} &=&
\frac{1}{L^r} \sum_{\Gamma \in {\cal Q}_{w,L,r}}
\frac{n(n-1)\ldots(n-v_\Gamma+1)}{\prod_{j=1}^k n(n-1)\ldots
(n-e^j_\Gamma+1)} = \\
&=& \frac{1}{L^r} \sum_{\Gamma \in {\cal Q}_{w,L,r}}
[n]_{v_\Gamma} \prod_{j=1}^k \frac{1}{[n]_{e^j_\Gamma}}  \nonumber
\end{eqnarray}

Note that the equality between \eqref{eq:moment-formula1} and
\eqref{eq:moment-formula2} holds only for $n$ large enough.
Indeed $N_\Gamma(n)=0$ if $n<e^j_\Gamma$
for some $\Gamma \in {\cal Q}_{w,L,r}$ and $j \in
\{1,\ldots,k\}$. This holds for $n \geq
max_{j=1,\ldots,k}\big(e_j(\widetilde{\Gamma}_{w,L,r})\big)$.

For every $L \geq 1$ and $r \geq 1$, \eqref{eq:moment-formula2}
thus yields a rational function in $n$, which, for sufficiently large $n$,
is the $r$-th factorial moment of $X_{w,L}^{(n)}$. It is
convenient to rewrite \eqref{eq:moment-formula2} as a function of
$\frac{1}{n}$:

\begin{equation} \label{eq:pre-psi}
\mathbb{E}\big([X_{w,L}^{(n)}]_r\big) = \frac{1}{L^r}\sum_{\Gamma
\in {\cal Q}_{w,L,r}} \left(\frac{1}{n}\right)^{e_\Gamma-v_\Gamma}
\frac{\prod_{t=1}^{v_\Gamma-1}(1-\frac{t}{n})}
{\prod_{j=1}^k\prod_{t=1}^{e^j_\Gamma-1}(1-\frac{t}{n})}
\end{equation}

We can now define a rational function $\psi_{w,L,r}$ by:
\begin{equation} \label{eq:psi}
\psi_{w,L,r}(x) = \frac{1}{L^r} \sum_{\Gamma\in{\cal Q}_{w,L,r}}
x^{\chi(\Gamma) - 1} \frac{\prod_{t=1}^{v_\Gamma-1}(1-tx)}
{\prod_{j=1}^k\prod_{t=1}^{e^j_\Gamma-1}(1-tx)}
\end{equation}
and so $\mathbb{E}\big([X_{w,L}^{(n)}]_r\big) =
\psi_{w,L,r}(\frac{1}{n})$ for $n \geq max_{j=1,\ldots,k}
\big(e_j(\widetilde{\Gamma}_{w,L,r})\big)$.

The following lemma shows that $\psi_{w,L,r}$ is well defined on a
neighborhood of $0$.

\begin{lem} \label{lem:edges}
For each $\Gamma \in {\cal Q}_{w,L,r}$, $\chi(\Gamma) \geq 1$.
\end{lem}

\begin{proof}
Note that in $\widetilde{\Gamma}_{w,L,r}$ every vertex has degree
2, since $w$ is cyclically reduced.
This degree can not decrease in a quotient. Therefore, every
vertex in every $\Gamma \in {\cal Q}_{w,L,r}$ has degree at least
2, and the lemma follows.
\end{proof}

It is well-known how to express the regular moments as linear
combinations of the factorial moments. Thus we can derive (an
efficiently computable) rational function $\Psi_{w,L,r}$ (which is
a linear combination of $\psi_{w,L,1},\ldots,\psi_{w,L,r}$), such
that $\mathbb{E}([X_{w,L}^{(n)}]^r)=\Psi_{w,L,r}(1/n)$ for
sufficiently large $n$. This function $\Psi_{w,L,r}$ is obviously
defined on a neighborhood of $0$, which proves the first part
of Corollary \ref{cor:moments}.

\subsection{Proving Theorem \ref{ther:limit-xwln} with the Method of
Moments}

The proof of Theorem \ref{ther:limit-xwln} is based on the method
of moments. Under certain mild conditions, a probability
distribution is determined by its moments, or as here, a limit
distribution is determined by the limits of the moments.

\subsubsection{Some Facts from the Method of Moments}

A probability measure $\mu$ on $\mathbb{R}$ is said to be
\emph{determined by its moments} if it has finite moments
$\alpha_r = \int_{-\infty}^\infty x^r\mu(dx)$ of all orders, and
$\mu$ is the only probability measure with these moments.
We quote Theorem 30.2 from \cite{bil95}:

\begin{ther} \label{metod-of-moments}
Let $X$ and $X_n ~(n \in \mathbb{N})$ be random variables, and
suppose that the distribution of $X$ is determined by its moments,
that the $X_n$ have moments of all order, and that ~$\lim_{n \to
\infty}\mathbb{E}(X_n^{\;r}) = \mathbb{E}(X^r)$ for
every $r \in \mathbb{N}$. \\
Then \[X_n \stackrel{dis}{\to} X.\]
\end{ther}
Note that if $X$ and the $X_n$ are integer-valued then $X_n
\stackrel{dis}{\to} X$ is equivalent to $Pr(X_n=k) \to Pr(X=k)$
for every integer $k$.

The relation between regular moments and factorial moments
implies:

\begin{cor} \label{cor:method-of-factorial-moments}
The statement of Theorem~\ref{metod-of-moments} holds where
moments are replaced by factorial moments.
\end{cor}

In this section we use Corollary
\ref{cor:method-of-factorial-moments} to prove Theorem
\ref{ther:limit-xwln}. That $X_{w,L}^{(n)}$ has moments of all
order is evident. We still need to show that the r.h.s of
\eqref{eq:conv} is determined by its moments, and that the $r$-th
factorial moment of $X_{w,L}^{(n)}$ indeed converges to the $r$-th
factorial moment of
this random variable. \\

Theorem 30.1 from \cite{bil95} provides a sufficient condition for
a probability measure $\mu$ to be determined by its moments,
namely, that the power series $\sum_r \alpha_r t^r / r!$ where
$\alpha_r$ is the $r$-th moment of $\mu$ has a positive radius of
convergence. (This series is the moment generating function of
$\mu$, when the latter exists.) For $\mu$ a Poisson distribution
(with any parameter), this power series converges for all real $t$
(e.g., \cite{pit}, section 4), so $\mu$ is determined by its
moments. A convolution (a summation) of several Poisson
distributions is itself Poisson (whose parameter is the sum of
parameters), and thus satisfies the condition as well.

In particular, recall that the r.h.s. of \eqref{eq:conv} is
$\sum_{h\in H(d,L)}hZ_{1/Lh}$, where $Z_{m} \sim Poi(m)$. If we
omit the constant $h$ of every term in this sum, we obtain
$\sum_{h\in H(d,L)}Z_{1/Lh}$, whose distribution is simply
$Poi\big( \sum_{h\in H(d,L)}1/Lh \big)$, which is determined by
its moments. According to the definition of $H(d,L)$ in
\eqref{eq:hdl}, each $h \in H(d,L)$ satisfies $1 \leq h \leq d$.
Thus, if we denote by $\alpha_r$ the $r$-th moment of $\sum_{h\in
H(d,L)}Z_{1/Lh} $ and by $\beta_r$ the $r$-th moment of
$\sum_{h\in H(d,L)}hZ_{1/Lh}$, then $\alpha_r \leq \beta_r \leq
d^r\alpha_r$. Consequently, the series $\sum_r \beta_r t^r / r!$
has radius of convergence that is $\ge \frac{1}{d}$
that of the series $\sum_r \alpha_r t^r / r!$.
But the latter converges for all real $t$,
hence so does $\sum_r \beta_r t^r
/ r!$, and the distribution of $\sum_{h\in
H(d,L)}hZ_{1/Lh}$ is determined by its moments.
\\
\\

To conclude the proof of Theorem \ref{ther:limit-xwln}, we need to
show that the (factorial) moments of $X_{w,L}^{(n)}$ indeed
converge to the respective moments of the r.h.s. of
\eqref{eq:conv}. But what is the limit of
$\mathbb{E}\big([X_{w,L}^{(n)}]_r\big)$? By Lemma \ref{lem:edges},
the limit of each term in the r.h.s. of \eqref{eq:pre-psi} is
either $0$ (if $\chi(\Gamma)\geq 2)$, or $1$ (if $\chi(\Gamma)=1$).
Therefore,

\begin{equation} \label{eq:lim=n_chi=1}
\lim_{n \to \infty} \mathbb{E}\big([X_{w,L}^{(n)}]_r\big) =
\frac{1}{L^r} \big|\{\Gamma \in {\cal Q}_{w,L,r}~:~ \chi(\Gamma)=1
\} \big|
\end{equation}

As explained in the proof of Lemma \ref{lem:edges}, the equality
$\chi(\Gamma)=1$ holds for some $\Gamma \in {\cal Q}_{w,L,r}$, iff
every vertex in $\Gamma$ has degree 2, i.e., iff $\Gamma$ is a
disjoint union of cycles.

We denote by ${\cal C}_{w,L,r}$ the subset of ${\cal Q}_{w,L,r}$
consisting of all graphs which are a disjoint union of cycles.
\eqref{eq:lim=n_chi=1} now becomes:

\begin{equation} \label{eq:lim=c}
\lim_{n \to \infty} \mathbb{E}\big([X_{w,L}^{(n)}]_r\big) =
\frac{1}{L^r}\big| {\cal C}_{w,L,r} \big|
\end{equation}

\subsubsection{The Graphs in ${\cal C}_{w,L,r}$} \label{sbsbs:cwlr}

Let $w \in \mathbf{F}_k$ be cyclically reduced and equal $u^d$
with $u$ primitive and $d \geq 1$. A graph $\Gamma \in {\cal
C}_{w,L,r}$ has a very specific structure: Each cycle $c$ in
$\Gamma$ must be a cyclic concatenation of several copies of $u$
(every cycle in $\Gamma$ looks like
$\widetilde{\Gamma}_{u,b,1}$ for some positive integer $b$).

To see this, recall that each cycle $c$ in $\Gamma$ represents a
closed trail through $w^L$ (at least one closed trail). Hence
there is some vertex $x$ in $c$ and some orientation on $c$, such
that if we leave $x$ in this orientation and go exactly $d \cdot
L$ times through $u$, we get back to $x$ (possibly after tracing
$c$ several times). Since $u$ is primitive, it cannot be invariant
under cyclic shift of any length $l < |u|$, and the size of $c$
must divide $|u|$. (In fact, the integer $|c|/|u|$ divides $d
\cdot L$).

Moreover, all closed trails that are represented in $c$, go in the
same direction (and thus also start in one of the $|c|/|u|$ head
vertices of $u$). We already saw in the proof of Lemma
\ref{lem:phi approx beta=2} that
for $u$ primitive, $u^{-1}$ is not a subword of $u^2$.
This rules out the possibility of ``finding $u$ in the opposite
direction''.

This analysis of the structure of the graphs in ${\cal C}_{w,L,r}$
yields an important corollary, which ultimately explains why the
limit distribution of $X_{w.L}^{(n)}$ depends solely on $d$ and
$L$, and not on $u$:

\begin{cor} \label{cor:u-not-matter}
Let $w_1,w_2 \in \mathbf{F}_k$ be cyclically reduced and equal
$u_1^{\;d}$ and $u_2^{\;d}$, respectively, with $u_1$ and $u_2$
primitive and $d \geq 1$. Then
\[ \big| {\cal C}_{w_1,L,r} \big| =
\big| {\cal C}_{w_2,L,r} \big| \]
\end{cor}

\begin{proof}
The analysis above shows that the inner structure of $u$ is
completely irrelevant to the graphs in ${\cal C}_{w,L,r}$, and
there is a natural bijection between $C_{w_1,L,r}$ and
$C_{w_2,L,r}$: simply replace each copy of $u_1$ by a copy of
$u_2$.
\end{proof}

\subsubsection{The Simple Case of a Primitive Word}

The material in this section is not necessary for the proof and
deals with the special case of primitive words. Our hope is that
this spacial case makes it easier to follow the general proof. So
let us assume that $u=w$ and $d=1$. The universal graph
$\widetilde{\Gamma}_{w,L,r}$ is a disjoint union of $r$ cycles,
each of which consists of $L$ copies of $u$. We have a total of $r
\cdot L$ copies of $u$, and their $r \cdot L$ initial vertices are
kept separated in each quotient. It is easily verified that the
only quotient of $\widetilde{\Gamma}_{w,L,r}$ which is a disjoint
union of cycles is $\widetilde{\Gamma}_{w,L,r}$ itself, and thus
$\big| {\cal C}_{w_1,L,r} \big| = 1$. By \eqref{eq:lim=c} we have
$\lim_{n \to \infty} \mathbb{E}\big([X_{w,L}^{(n)}]_r\big) =
\frac{1}{L^r}$.

On the other hand, when $w$ is primitive, the r.h.s. of
\eqref{eq:conv} is simply $Z_{1/L}$. Let $X$ be an integer-valued
non-negative random variable and let $f_X(t)$ be its generating function \[f_X(t)
= \sum_{j=0}^{\infty} Pr(X=j)t^j .\] Under certain mild conditions
and in particular if $f_X(t)$ is analytic on $\mathbb{R}$,
\[ \mathbb{E}([X]_r) = f_X^{\;(r)}(1) \qquad \forall r\in\mathbb{N}. \]
The generating function of $Z_{1/L}$ is $f(t) =
e^{\frac{-1}L}e^{\frac{t}L}$. Thus,
\[ \mathbb{E}([Z_{1/L}]_r) = f^{(r)}(1) = \frac{1}{L^r} \]
which completes the proof of Theorem \ref{ther:limit-xwln} for $w$
primitive.

\subsubsection{The General Case} \label{end-of-proof}

To complete the proof of Theorem \ref{ther:limit-xwln}, we show
that for every $w=u^d \in \mathbf{F}_k$, $L \geq 1$ and $r \geq
1$, the series $\mathbb{E}\big([X_{w,L}^{(n)}]_r\big)$ converges,
for $n \to \infty$, to $\mathbb{E}\big([Y_{d,L}]_r\big)$, where
$Y_{d,L}$ is the r.h.s. of \eqref{eq:conv}, i.e.
$Y_{d,L}=\sum_{h\in H(d,L)}hZ_{1/Lh}$.

Let $X_1$ and $X_2$ be non-negative integer-valued variables with
generating functions $f_{X_1}(t), f_{X_2}(t)$. If $Y=X_1+X_2$,
then clearly $Y$'s generating function is $f_Y(t)=f_{X_1}(t) \cdot
f_{X_2}(t)$. The generating function of $hZ_m$ ($h \in \mathbb{N},
Z_m \sim Poi(m)$) is $f_{hZ_m}(t) = e^{-m} \cdot e^{mt^h}$.
Therefore,
\[ f_{Y_{d,L}} = \prod_{h \in H(d,L)}
e^{-\frac{1}{Lh}}e^{\frac{t^h}{Lh}} \] and if $H(d,L) =
\{h_1,\ldots,h_p\}$, then
\begin{eqnarray} \label{eq:derivative}
\mathbb{E}\big([Y_{d,L}]_r\big) &=&
f_{Y_{d,L}}^{\;(r)}(t)\big|_{t=1} \nonumber \\ &=&
\sum_{\substack{r_1 + \ldots + r_p = r \\ r_j \geq
0}}\binom{r}{r_1 \ldots r_p} \prod_{j=1}^p f_{h_jZ_{1/Lh_j}}
^{~~~~(r_j)}(t)\big|_{t=1}
\end{eqnarray}

By \eqref{eq:lim=c}, the series
$\mathbb{E}\big([X_{w,L}^{(n)}]_r\big)$ converges to
$\frac{1}{L^r}\big| {\cal C}_{w,L,r} \big|$, so we proceed to
analyze the set ${\cal C}_{w,L,r}$. Let $\Gamma$ be a graph in
${\cal C}_{w,L,r}$. As explained in section \ref{sbsbs:cwlr},
$\Gamma$ is a disjoint union of cycles, each of which consists of
several copies of $u$. If we let $b$ denote the number of copies
of $u$ in some cycle, what are the possible values of $b$? To
begin with, $b|dL$, as this cycle represents a closed path that
consists of $dL$ copies of $u$. Secondly, as this cycle represents
an $L$-cycle, $L|b$ and $b \nmid dL'$ for any $1 \leq L' < L$. If
we let $h=\frac{b}{L}$, then the constraints on $b$ translate into
the following conditions: $h|d$ and
$\left(\frac{d}{h},L\right)=1$, precisely as in the definition of
$H(d,L)$. Note that a cycle in $\Gamma$ with $b=Lh$ copies of $u$
can represent up to $h$ distinct $L$-cycles in
$w(\sigma_1,\ldots,\sigma_k)$.

As before, let $H(d,L)= \{h_1,\ldots,h_p\}$. For $j=1,\ldots,p$
consider those $L$-cycles which are associated in the quotient
$\Gamma$ to a cycle of length $Lh_j \cdot |u|$. (The $i$-th
$L$-cycle belongs to the cycle $c$ in $\Gamma$ if the blocks
containing $s^i_0,\ldots,s^i_{L|w|-1}$ correspond to vertices in
$c$.) Let $r_j$ be the number of such $L$-cycles whence $\sum r_j
= r$, and there are $\binom{r}{r_1 \ldots r_p}$ ways to choose
which $L$-cycles go where (Recall that
$\widetilde{\Gamma}_{w,L,r}$ consists of an \emph{ordered} list of
$r$ $L$-cycles). Now let ${\cal C}_{w,L,r}^h$ denote the subset of
${\cal C}_{w,L,r}$ consisting of all quotient graphs where all
disjoint cycles are of equal length of $Lh|u|$ each. Then we have:
\begin{eqnarray} \label{eq:cwlrh}
\lim_{n \to \infty} \mathbb{E}\big([X_{w,L}^{(n)}]_r\big)
&\stackrel{\eqref{eq:lim=c}}{=}& \frac{1}{L^r}\big| {\cal
C}_{w,L,r} \big| = \frac{1}{L^r} \sum_{\substack{r_1 + \ldots +
r_p = r \\ r_j \geq 0}} \binom{r}{r_1 \ldots r_p} \prod_{j=1}^p
\big| {\cal C}_{w,L,r_j}^{h_j} \big|
\nonumber \\
&=& \sum_{\substack{r_1 + \ldots + r_p = r \\ r_j \geq 0}}
\binom{r}{r_1 \ldots r_p} \prod_{j=1}^p \frac{1}{L^{r_j}}\big|
{\cal C}_{w,L,r_j}^{h_j} \big|
\end{eqnarray}
(${\cal C}_{w,L,0}^{h}$ denotes the singleton of the empty graph, and
therefore $\big| {\cal C}_{w,L,0}^{h} \big| = 1$).

By combining \eqref{eq:derivative} and \eqref{eq:cwlrh}, we
conclude that Theorem \ref{ther:limit-xwln} will follow if we show
\begin{equation} \label{eq:left}
\frac{1}{L^{r}}\big| {\cal C}_{w,L,r}^{h} \big| = f_{hZ_{1/Lh}}
^{~~~(r)}(t)\big|_{t=1},
\end{equation}
for every $L \geq 1, r \geq 0$ and $h \in H(d,L)$. \\ \\

We begin with the l.h.s. of \eqref{eq:left}. Recall that by
definition, each graph $\Gamma \in {\cal C}_{w,L,r}^{h}$ consists
of a disjoint union of cycles of length $Lh|u|$ each. Thus, each
cycle represents up to $h$ distinct $L$-cycles of
$w(\sigma_1,\ldots,\sigma_k)$, and $\Gamma$ can represent up to $h
\cdot (\textrm{\# cycles in }\Gamma)$ distinct $L$-cycles. But
$\Gamma$ represents only $r$ distinct $L$-cycles, whence there are
$h \cdot (\textrm{\# cycles in }\Gamma) - r$ ``free spots'' in
$\Gamma$ that can contain new $L$-cycles. We illustrate these notions
in Figure \ref{fig:c_hwlr_example}.

\begin{figure}[htb]
\centering
\includegraphics[width=1\textwidth]{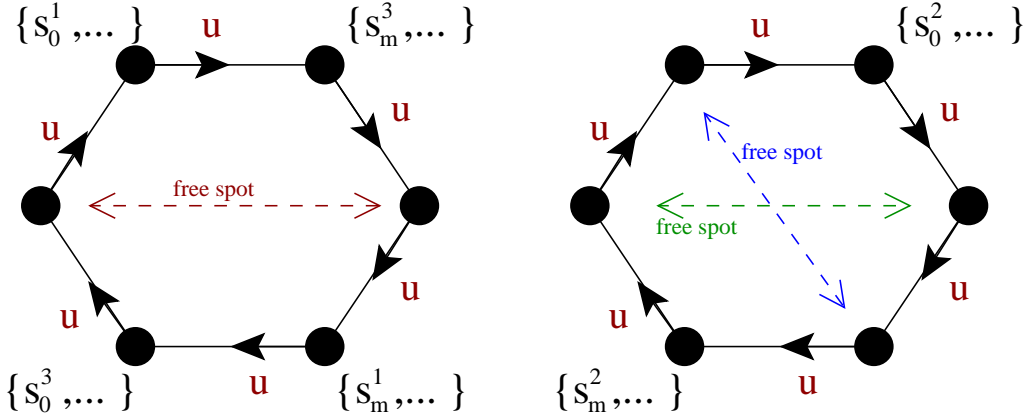}
\caption{One of the graphs $\Gamma \in {\cal C}_{w,L,r}^{h}$, where
$w=u^3$ for some primitive $u \in \mathbf{F}_k$, $h=3$, $L=2$ and
$r=3$. (We use the labels $s^j_i$ introduced in Section
\ref{sbs:categories}.) This $\Gamma$ contains two cycles, each of
which consists of six copies of $u$. Each of the cycles can correspond
to up to three distinct $L$-cycles (indeed, $h=3$), so $\Gamma$ can contain
up to six $L$-cycles. As it already contains three, it has $6-3=3$ free
spots.}  \label{fig:c_hwlr_example}
\end{figure}

Now let $\alpha_{w,L,r}^{h}[j]$ denote the number of graphs in ${\cal
C}_{w,L,r}^{h}$ with $j$ free spots. We can define the
generating function of the $\alpha_{w,L,r}^{h}[j]$:
\[g_{w,L,r}^h(t) = \alpha_{w,L,r}^{h}[0] + \alpha_{w,L,r}^{h}[1]t +
\alpha_{w,L,r}^{h}[2]t^2 + \ldots
\]
and obviously $g_{w,L,r}^h(1) = \big| {\cal C}_{w,L,r}^{h} \big|$.

Before we derive a recursion formula for this function, we want to
illustrate by writing explicit expressions for $r=0,1,2$.
Every connected component (=cycle) in every
$\Gamma \in {\cal C}_{w,L,r}^{h}$, realizes at least one
$L$-cycle. Thus, when $r=0$ and there are no $L$-cycles at all, we
have only the empty graph which has no free spots, and so
$g_{w,L,0}^h(t) = 1$. When $r=1$, we have a single graph in ${\cal
C}_{w,L,1}^{h}$, with $h-1$ free spots (a single cycle of $Lh$
copies of $u$), and therefore $g_{w,L,1}^h(t) = t^{h-1}$. For $r=2$
there is always a two-cycle graph with one $L$-cycle in each cycle.
This graph has $2(h-1)$ free spots. In addition, if $h \geq 2$, there
are also graphs consisting of a single cycles that corresponds to two
$L$-cycles. There are $L(h-1)$ ways to place the two $L$-cycles and such
graphs have $(h-2)$ free spots. Thus, for $h=1$, $g_{w,L,2}^h(t) = t^{2(h-1)}$
and for $h \geq 2$ $g_{w,L,2}^h(t) = L(h-1)t^{h-2} + t^{2(h-1)}$.

We now want to derive the functions $g_{w,L,r}^h$ by recursing
on $r$. Let $\Gamma$ be a graph in ${\cal C}_{w,L,r}^{h}$ with $j$ free
spots. In what manners can we add another $L$-cycle and make it a graph in
${\cal C}_{w,L,r+1}^{h}$? We have two options: we can put the new
$L$-cycle in one of the $j$ free spots, in $L$ possible ways ($L$
possible cyclic shifts), resulting in $j \cdot L$ different graphs
in ${\cal C}_{w,L,r+1}^{h}$, each of which has $j-1$ free spots.
Alternatively, we can add one new cycle to $\Gamma$ and put there
our new $L$-cycle, which yields a single graph in ${\cal
C}_{w,L,r+1}^{h}$ with $j+h-1$ free spots. Thus, we have:
\[g_{w,L,r+1}^h(t) = L \cdot \left(g_{w,L,r}^h(t)\right)' +
t^{h-1} \cdot g_{w,L,r}^h(t). \]

We now go back to the right side of \eqref{eq:left}. Recall that
$f_{hZ_{1/Lh}}(t) = e^{-\frac{1}{Lh}}e^{\frac{t^h}{Lh}}$. If we
write $f_{hZ_{1/Lh}} ^{~~~(r)}(t) =
e^{-\frac{1}{Lh}}e^{\frac{t^h}{Lh}} \cdot q_{L,r}^h(t)$ where
$q_{L,r}^h(t)$ is the appropriate polynomial, then
\[q_{L,0}^h(t) = 1 \]
and \[q_{L,r+1}^h(t) = \left(q_{L,r}^h(t)\right)' + \frac{1}{L}
\cdot t^{h-1} \cdot q_{L,r}^h(t) \]

Thus $g_{w,L,r}^h(t) = L^r \cdot q_{L,r}^h(t)$, and we can
conclude:

\begin{equation}
\frac{1}{L^{r}}\big| {\cal C}_{w,L,r}^{h} \big| = \frac{1}{L^r}
\cdot g_{w,L,r}^h(1) = q_{L,r}^h(1) = f_{hZ_{1/Lh}} ^{~~~(r)}(1)
\end{equation}
when the last equality comes from the fact that
$e^{-\frac{1}{Lh}}e^{\frac{t^h}{Lh}}\big|_{t=1}=1$. \\ \\

The proof of Theorem \ref{ther:limit-xwln} is now complete.
Corollary \ref{cor:moments} is also proved. For completeness sake,
here is a short proof of the last sentence in the corollary,
namely, that $\lim_{n \to
\infty}\mathbb{E}(X_{w,L}^{(n)})=\frac{|H(d,L)|}{L}$.

\begin{proof}
\begin{eqnarray*}
\lim_{n \to \infty}\mathbb{E}(X_{w,L}^{(n)}) &=&
\mathbb{E}(Y_{d,L}) = \mathbb{E}\left( \sum_{h\in H(d,L)}hZ_{1/Lh}
\right) = \sum_{h\in H(d,L)} \mathbb{E}\left( hZ_{1/Lh} \right) \\
&=& \sum_{h\in H(d,L)} h \cdot \frac{1}{Lh} = \frac{|H(d,L)|}{L}
\end{eqnarray*}
\end{proof}

\begin{remark}
In fact, the technique presented here are likely to yield
further results. The
method of moment applies as well to random vectors (for
distributions that are determined by their moments, see,
e.g., \cite{jlr}, Theorem 6.2). The joint moments of
$X_{w,L_1}^{(n)}, \ldots, X_{w,L_k}^{(n)}$ for some $k$ and
positive integers $L_1, \ldots, L_k$ can be analyzed similarly to
the way we analyzed the moments of $X_{w,L}^{(n)}$ for some $L$,
and the limit joint distribution of these variables is probably
determined by its moments.
\end{remark}
\begin{remark}
It is of great interest to study word maps for other groups or
for non-uniform distributions on $S_n$. For results of this nature
see \cite{ben06}.
\end{remark}

\section{Open problems}
\label{sec:open_prob}

Many interesting questions and conjectures were raised in this
paper. We collect them here.

\begin{itemize}
\item
Let $u$ and $w$ be two words such that for any finite group $G$, the
distribution of the two word maps on $G$ are identical. Is it true
that $u \sim w$?
\item
(Conjecture \ref{conj:beta=phi}) $\beta(w) = \phi(w)$ for every
word $w$.
\item
(A consequence of Conjecture~\ref{conj:beta=phi}:) For every word $w$,
and sufficiently large $n$, a random permutation in the image of $w$ in $S_n$
has, on average, at least one fixed point.
\item
Friedman's Conjecture: For every base graph $G$, almost surely all
new eigenvalues in lifts of $G$ are $\le \rho + o(1)$.
\item
Nica's theorem determines the behavior of the number of $L$-cycles
in the $S_n$-image of any formal word $w$. There are numerous
other parameters of such permutations (e.g. the number of cycles)
whose typical behavior is still not understood.
\end{itemize}


\begin{thebibliography}{9999999999}

\bibitem[Abe]{abert}
M. Ab\'{e}rt. On the Probability of Satisfying a Word in a Group.
\emph{Journal of Group Theory}, to appear.

\bibitem[AL06]{al}
A. Amit and N. Linial. Random Lifts of Graphs II: Edge Expansion.
\emph{Combinatorics Probability and Computing}, 15(2006) 317-332.

\bibitem[ALM02]{alm}
A. Amit, N. Linial and J. Matousek. Random Lifts of Graphs III:
Independence and Chromatic Number. \emph{Random Structures and
Algorithms}, 20(2002) 1-22.

\bibitem[Ben06]{ben06}
F. Benaych-Georges. Cycles of Free Words in Several Independant
Random Permutations with Restricted Cycle Lengths.
\emph{arXiv/abs/math/0611500}

\bibitem[Bil95]{bil95}
P. Billingsley. \emph{Probability and Measure}.
Wiley-Interscience, New York, third edition, 1995.

\bibitem[BL06]{bl}
Y. Bilu and N. Linial. Lifts, discrepancy and nearly optimal
spectral gaps. \emph{Combinatorica}, 26(2006), 495-519.

\bibitem[BS87]{bs}
A. Broder and E. Shamir. On the Second Eigenvalue of random
regular graphs. In \emph{The 28th Annual Symposium on Foundations
of Computer Science}, pp. 286-294, 1987.

\bibitem[Buc86]{buc}
M. W. Buck. Expanders and Diffusers, \emph{SIAM J. Algebraic
Discrete Methods}, 7 (1986), 282-304. MR 87e:68088

\bibitem[Car72]{car}
P. Cartier. Fonctions Harmonique sur un arbre. In \emph{Symposia
Mathematica}, Vol. IX (\emph{Convegno di calcolo delle
Probabilita, INDAM, Rome, 1971}), pages 203-270. Academic Press,
London, 1972. MR0353467 (50:5950)

\bibitem[Fri91]{fri91}
J. Friedman. On the Second Eigenvalue and Random Walks in Random
$d$-Regular Graphs. \emph{Combinatorica}, 11(4):331-362, 1991.

\bibitem[Fri03]{fri03}
J. Friedman. Relative Expanders or weakly Relatively Ramanujan
Graphs. \emph{Duke Math. J.}, 118(1), 19-35, 2003.

\bibitem[Fri]{fri04}
J. Friedman. A proof of Alon's second eigenvalue conjecture.
Memoirs of the A.M.S., to appear.

\bibitem[Gre95]{gl}
Y. Greenberg. \emph{On the Spectrum of Graphs and their Universal
Coverings}. PhD thesis, Hebrew University of Jerusalem, 1995 (in
Hebrew).

\bibitem[HLW06]{hlw}
S. Hoory, N. Linial and A. Widgerson. Expander Graphs and their
Applications. In \emph{Bulletin of the American Mathematical
Society}, Volume 43, Numer 4, October 2006, pp. 439-531.

\bibitem[J\L R00]{jlr}
S. Janson, T. {\L}uczak, A. Rucinski. \emph{Random Graphs}.
Wiley-Interscience, New York, 2000.

\bibitem[LR05]{lr}
N. Linial, and E. Rozenman. Random Lifts of Graphs: Perfect
Matchings. \emph{Combinatorica}, 25(2005) 407 - 424.

\bibitem[LSh07]{lsh}
M. Larsen and A. Shalev. Word Maps and Waring Type Problems.
Preprint, to appear.

\bibitem[MKS66]{mks}
W. Magnus, A. Karrass, and D. Solitar. \emph{Combinatorial Group Theory},
Dover, New York, 1966.

\bibitem[Nica94]{nica}
A. Nica. On the Number of Cycles of Given Length of a Free Word in
Several Random Permutations. \emph{Random Structures \&
Algorithms} 5(1994), 703-730.

\bibitem[Nil91]{nil}
A. Nilli. On the Second Eigenvalue of a Graph. \emph{Discrete
Math.}, 91(2):207-210, 1991. Mr1124768 (92j:05124)

\bibitem[Pit97]{pit}
J.Pitman. Some Probabilistic Aspects of Set Partitions. \emph{The
American Mathematical Monthly}, Vol. 104, No. 3 (Mar., 1997), pp.
201-209.

\bibitem[Wig55]{wig}
E. Wigner. Characteristic Vectors of Bordered Matrices with
Infinite Dimensions. \emph{Ann. of Math}, 62, 548-564, 1955.

\end{thebibliography}
\end{document}